\documentclass[a4paper,11pt,reqno,twoside]{amsart}

\usepackage{amsmath}
\usepackage{amsfonts}
\usepackage{amssymb}
\usepackage{amsthm}
\usepackage{color}
\usepackage{ifpdf}
\usepackage{array}
\usepackage{mathtools}
\usepackage{url}
\usepackage{multirow,bigdelim}

\numberwithin{equation}{section}
 
\newtheorem{theorem}{Theorem}[section] 
\newtheorem{lemma}{Lemma}[section] 
\newtheorem{proposition}{Proposition}[section]

\newcommand*{\C}{\mathbb{C}}
\newcommand*{\R}{\mathbb{R}}
\newcommand*{\Q}{\mathbb{Q}}
\newcommand*{\Z}{\mathbb{Z}}
\newcommand*{\N}{\mathbb{N}}

\newcommand{\comment}[1]{}
\title[Hamiltonians arising from $L$-functions in the Selberg class]%
      {Hamiltonians arising from \\ $L$-functions in the Selberg class} 
\author[M. Suzuki]{Masatoshi Suzuki}
\AtBeginDocument{%
\begin{abstract}
We establish a new equivalent 
condition for the Grand Riemann Hypothesis 
for $L$-functions in a wide subclass of the Selberg class
in terms of canonical systems of differential equations. 
A canonical system is determined by a real symmetric matrix-valued function called a Hamiltonian.  
To establish the equivalent condition, 
we use an inverse problem for canonical systems of a special type. 
\end{abstract}
\maketitle
}
\begin{document}

%
%
\section{Introduction} \label{section_1} 
%
%

The Riemann Hypothesis (RH) asserts that all nontrivial zeros of the Riemann zeta-function $\zeta(s)$ lie on the critical line $\Re(s)=1/2$, 
and it has been generalized to wider classes of zeta-like functions. 
Especially, an analogue of RH for $L$-functions in the Selberg class is often called the Grand Riemann Hypothesis (GRH)\footnote{
The abbreviation GRH is often used for the Generalized Riemann Hypothesis in literature 
but we use it for the Grand Riemann Hypothesis throughout this paper.}. 

The present paper aims to establish a new equivalent condition for 
GRH for $L$-functions in a wide subclass of the Selberg class 
in terms of canonical systems by using general results in the preliminary paper \cite{Su19_1} 
studying an inverse problem of canonical systems of a special type. 
We explain the relation by dealing with the case of the Riemann zeta function in the introduction. 
The present study was mainly stimulated by the works of 
Lagarias~\cite{La05, La06, La09} 
(and also Burnol \cite{Burnol2002, Burnol2004, Burnol2007, Burnol2011}
as well as \cite{Su19_1}). 
\medskip

The Riemann xi-function 
\begin{equation*}
\xi(s)=\frac{1}{2}s(s-1)\pi^{-s/2}\Gamma\left(\frac{s}{2}\right)\zeta(s)
\end{equation*}
is an entire function taking real-values on the critical line such that 
the zeros coincide with nontrivial zeros of $\zeta(s)$. 
Therefore RH is equivalent that all zeros of 
the entire function $\xi(s)$ lie on the critical line.  
Noting this, we start with considering an additive decomposition 
\begin{equation} \label{s101}
\xi\Bigl(\frac{1}{2}-iz\Bigr) = \frac{1}{2}(E(z)+E^\sharp(z))
\end{equation}
by an entire function $E$, where 
$i=\sqrt{-1}$, $F^\sharp(z)=\overline{F(\bar{z})}$ 
for an entire function $F$ 
and the bar stands for the complex conjugate. 
It is easily confirmed that \eqref{s101} holds for 
\begin{equation} \label{s102}
E_\xi(z)= 
\left( 
\xi(s) + \left.
\frac{d}{ds}\xi(s)\right)\right|_{\displaystyle{s=\tfrac{1}{2}-iz}}. 
\end{equation}
by the functional equations $\xi(s)=\xi(1-s)$ and $\xi(s)=\xi^\sharp(s)$
and thus \eqref{s101}  also holds for infinitely many entire functions $E=E_\xi+F$, 
where 
$F$ is an entire function satisfying $F^\sharp=-F$. 

The advantage of the decomposition \eqref{s101} stands on 
the theory of the {\it Hermite--Biehler class} of entire functions. 
We denote by $\overline{\mathbb{HB}}$ the set of all entire functions satisfying 
\begin{equation} \label{s103}
|E^\sharp(z)| < |E(z)|  \quad \text{for all $z \in \C_+$},
\end{equation}
where $\C_+=\{z=u+iv\,|\, u,v \in \R,\, v>0\}$. 
The Hermite--Biehler class $\mathbb{HB}$ 
consists of all $E \in \overline{\mathbb{HB}}$ having no real zeros. 
If $E \in \overline{\mathbb{HB}}$, the two entire functions 
\begin{equation} \label{s104}
A(z) := \frac{1}{2}(E(z)+E^\sharp(z)) \quad \text{and} \quad 
B(z) := \frac{i}{2}(E(z)-E^\sharp(z)) 
\end{equation}
have only real zeros. 
Further, all (real) zeros of $A$ and $B$ are simple 
if $E \in \mathbb{HB}$ (\cite{MR0114002}). 
Therefore, 
the existence of $E \in \mathbb{HB}$ satisfying \eqref{s101} 
implies that RH holds together with the Simplicity Conjecture (SC) 
which asserts that all nontrivial zeros of $\zeta(s)$ are simple. 
Conversely, 
there exists $E \in \mathbb{HB}$ satisfying \eqref{s101} if we assume that RH and SC hold. 
In fact, $E_\xi$ of \eqref{s102} belongs to $\mathbb{HB}$ under RH and SC (\cite{La06}). 
\smallskip

The above discussion suggests the following strategy to the proof of RH and SC: 
first, find an entire function $E$ satisfying \eqref{s101};  
second, prove that $E$ belongs to $\mathbb{HB}$. 
Then these two conditions conclude RH and SC.  
More simply, 
we may start from the second step by using $E_\xi$ of \eqref{s102}. 
The obvious difficulty with this strategy is in the second step, 
because the reason why an entire function $E$ satisfying \eqref{s101} (or $E_\xi$) 
belongs to $\mathbb{HB}$ is not known other than RH and SC, 
so we do not know how to prove $E \in \mathbb{HB}$ without RH and SC. 
In this paper, we propose a strategy to prove $E \in \mathbb{HB}$ 
with the help of de Branges' theory of canonical systems 
as with Lagarias \cite{La05, La06, La09} 
in which the applicability of the theory of de Branges  
to the study of the zeros of $L$-functions is suggested. 
\medskip

To start with, 
we explain how canonical systems generate functions of the class $\overline{\mathbb{HB}}$. 
A $2 \times 2$ real symmetric matrix-valued function $H$ defined on     
$I=[t_0,t_1)$ $(-\infty<t_0<t_1 \leq \infty)$ is called a {\it Hamiltonian} 
if  $H(t)$ is positive semidefinite for almost every $t \in I$, 
$H\not\equiv 0$ on any subset of $I$ with positive Lebesgue measure, 
and $H$ is locally integrable on $I$ with respect to the Lebesgue measure. 
The first-order system 
\begin{equation} \label{s105} 
-\frac{d}{dt}
\begin{bmatrix}
A(t,z) \\ B(t,z)
\end{bmatrix}
= z 
\begin{bmatrix}
0 & -1 \\ 1 & 0
\end{bmatrix}
H(t)
\begin{bmatrix}
A(t,z) \\ B(t,z)
\end{bmatrix}, \quad 
z \in \C
\end{equation}
associated with a Hamiltonian $H$ on $I$ is called a {\it canonical system} on $I$ 
(See \cite[Section 1]{Su19_1} for a difference with the usual definition).  
For a solution ${}^{\rm t}[A(t,z),B(t,z)]$ of a canonical system, we define  
\begin{equation} \label{s106}
J(t;z,w) 
 := \frac{\overline{A(t,z)}B(t,w)-A(t,w)\overline{B(t,z)}}{\pi(w-\bar{z})}.
\end{equation}
Then, $E(t,z):=A(t,z)-iB(t,z)$ is an entire function of $\overline{\mathbb{HB}}$ for every  $t \in I$ 
if we suppose that 
$\lim_{t \to t_1}J(t;z,z)=0$ for every $z \in \C_+$, 
unless $\det H(t)=0$ for almost every $t \in I$ 
(see \cite[Proposition 2.4]{Su19_1} for details). 
Conversely, for $E \in \overline{\mathbb{HB}}$, 
there exists a Hamiltonian $H$ defined on a (possibly unbounded) interval $I=[t_0,t_1)$ 
such that for a unique solution ${}^{\rm t}[A(t,z),B(t,z)]$ of  
\eqref{s105} satisfying  $E(z)=A(t_0,z)-iB(t_0,z)$, 
the function 
$E(t,z)=A(t,z)-iB(t,z)$ belongs to $\overline{\mathbb{HB}}$ for every $t \in I$, 
and $J(t;z,w)$ defined by \eqref{s106} satisfies 
$\lim_{t \to t_1}J(t;z,z)=0$ for every $z \in \C_+$ 
(\cite[Theorem 40]{MR0229011}). 
\smallskip

From the above, if we assume that RH and SC hold, 
there exists a Hamiltonian $H_\xi$ defined on some interval $I$ 
such that $E_\xi$ of \eqref{s102} is recovered from a solution of the canonical system 
associated with $H_\xi$. 
Using the conjectural Hamiltonian $H_\xi$, 
the above naive strategy to the proof of RH 
is now refined as follows: 
\begin{enumerate}
\item[(1-1)] Constructing the Hamiltonian $H_\xi$ on $I=[t_0,t_1)$ without RH; 
\item[(1-2)] Constructing the solution ${}^{\rm t}[A_\xi(t,z),B_\xi(t,z)]$ 
of the canonical system on $I$ associated with $H_\xi$ 
satisfying $E_\xi(z)=A_\xi(t_0,z)-iB_\xi(t_0,z)$; 
\item[(1-3)] Showing that $\lim_{t \to t_1}J_\xi(t;z,z)=0$ for every $z \in \C_+$, 
where $J_\xi(t;z,w)$ is the function defined by \eqref{s106} for ${}^{\rm t}[A_\xi(t,z),B_\xi(t,z)]$. 
\end{enumerate}

We can conclude that $E_\xi$ belongs to $\overline{\mathbb{HB}}$
if the above three steps are completed, 
and hence all zeros of $\xi(1/2-iz)=A_\xi(z)=(E_\xi(z)+E_\xi^\sharp(z))/2$ 
are real; this is nothing but RH. 
However, this approach faces a serious obstacle from the first step, 
because the inverse problem of canonical systems, 
which explicitly recovers $H$ from a given $E \in \mathbb{HB}$, 
is difficult in general. 
As mentioned in the introduction to \cite{Su19_1}, 
there are several methods for constructing a Hamiltonian $H$ 
for a large class of $E \in \overline{\mathbb{HB}}$, 
but they are inconvenient to use in the above strategy, 
because such methods are usually not applicable 
if it is not known whether $E \in \overline{\mathbb{HB}}$. 

In order to avoid such obstacles, we consider the family of entire functions
\begin{equation*}
E_\xi^\omega(z) := \xi(\tfrac{1}{2}+\omega-iz), \quad \omega >0 
\end{equation*}
instead of the single function $E_\xi$. 
(Note that $E_\xi\not=E_\xi^0$, but $E_\xi^\omega$ for small $\omega>0$ is similar to $E_\xi$ 
in the sense that 
$E_\xi^\omega(z)= \xi(\tfrac{1}{2}+\omega-iz) + \omega\,\xi'(\tfrac{1}{2}+\omega-iz) +O(\omega^2)$ 
for small $\omega>0$ if $z$ in a compact set, 
where ``$O$'' is the Landau symbol.) 
Then we find that a necessary and sufficient condition for RH (not require SC) 
is that $E_\xi^\omega \in \mathbb{HB}$ for every $\omega>0$ 
(Propositions \ref{prop_2_1} and \ref{prop_2_2}). 
In particular, there exists a Hamiltonian $H_\xi^\omega$ for every $\omega>0$ under RH. 
Therefore, RH is proved by completing the following three steps for every $\omega>0$: 
\begin{enumerate}
\item[(2-1)] Constructing the Hamiltonian $H_\xi^\omega$ on $I=[t_0,t_1)$ without RH; 
\item[(2-2)] Constructing the solution ${}^{\rm t}[A_\xi^\omega(t,z),B_\xi^\omega(t,z)]$ of the canonical system on $I$ associated with $H_\xi^\omega$ 
satisfying $E_\xi^\omega(z)=A_\xi^\omega(t_0,z)-iB_\xi^\omega(t_0,z)$; 
\item[(2-3)] Showing that $\lim_{t \to t_1}J_\xi^\omega(t;z,z)=0$ for all $z \in \C_+$. 
\end{enumerate}

There are two advantages to the second strategy. 
The first advantage is that 
a Hamiltonian $H_\xi^\omega$ on $I=[0,\infty)$ and 
a solution ${}^{\rm t}[A_\xi^\omega(t,z),B_\xi^\omega(t,z)]$ 
of the associate canonical system 
are explicitly constructed in \cite{Su12} 
under the restriction $\omega>1$ 
by applying the method of Burnol \cite{Burnol2011} introduced for the study of the Hankel transform. 
The second  advantage is the avoiding of SC 
and the central zero, that is, multiple zeros on the critical line 
and the zero at the central point $s=1/2$ are allowed to $\xi$ in the second strategy. 
This point is important to generalize the above strategy to the other $L$-functions, 
because they often have a multiple zero at the central point $s=1/2$.  

In \cite{Su12}, $H_\xi^\omega$ and ${}^{\rm t}[A_\xi^\omega(t,z),B_\xi^\omega(t,z)]$  
are constructed by using solutions $\varphi_t^\pm$ of the integral equations
%
$
\varphi_t^\pm(x) \pm \int_{-\infty}^{t} K_\xi^\omega(x+y) \varphi_t^\pm(y) \, dy = K_\xi^\omega(x+t), 
$
%
where $K_\xi^\omega$ is the kernel defined by 
\begin{equation*}
K_\xi^\omega(x) = \frac{1}{2\pi} \int_{\Im(z)=c} \frac{(E_\xi^\omega)^\sharp(z)}{E_\xi^\omega(z)} e^{-izx} \, dz  
\end{equation*}
for large $c>0$. 
The behavior of ${}^{\rm t}[A_\xi^\omega(t,z),B_\xi^\omega(t,z)]$ at $t = \infty$ and its role 
in the proof of $E_\xi^\omega \in \mathbb{HB}$ 
were not studied in \cite{Su12}.   
However, if the construction of $H_\xi^\omega$ 
and ${}^{\rm t}[A_\xi^\omega(t,z),B_\xi^\omega(t,z)]$ is extended to $0<\omega \leq 1$ 
together with an additional reasonable result 
on the behavior of  ${}^{\rm t}[A_\xi^\omega(t,z),B_\xi^\omega(t,z)]$ at $t=\infty$,  
we obtain $E_\xi^\omega \in \mathbb{HB}$ for every $\omega>0$, 
which implies RH. 
Unfortunately, there were several technical difficulties in \cite{Su12} 
to extend the construction of $H_\xi^\omega$ and 
${}^{\rm t}[A_\xi^\omega(t,z),B_\xi^\omega(t,z)]$  to $0<\omega \leq 1$. 
For instance, $K_L^\omega$ is far from continuous functions and $L^2$-functions if $\omega>0$ is small, 
and this fact is a serious obstacle for the construction in \cite{Su12}. 
\medskip

In this paper, we resolve the above technical difficulties 
by introducing the additional discrete parameter $\nu$:
\begin{equation*}
E_\xi^{\omega,\nu}(z) = \xi(\tfrac{1}{2}+\omega-iz)^\nu, \quad \omega \in \R_{>0}, \quad \nu \in \Z_{>0}.
\end{equation*}
The parameter $\nu$ does not affect to study 
whether $E_{L}^{\omega,\nu} \in \mathbb{HB}$ by definition of $\mathbb{HB}$, 
but plays an important role in constructing 
the conjectural Hamiltonian $H_\xi^{\omega,\nu}$ 
obtained by applying the method in \cite{Su19_1} to $E_\xi^{\omega,\nu}$. 
We will show in Section \ref{section_4} 
that $E_L^{\omega,\nu}$ satisfies all conditions required for the method in \cite{Su19_1} 
due to the existence of parameter $\nu$. 
In this way, a large part of (2-1) and (2-2) are achieved successfully for each $\omega>0$. 
To study (2-3), we need more preparation, as described in Section \ref{section_2}.
\medskip

Summarizing the above discussion, we will obtain an equivalent condition for RH 
in terms of canonical systems associated with Hamiltonians. 
This framework to establish the equivalent condition for RH 
applies to more general zeta and $L$-functions. 
We apply it to $L$-functions in the Selberg class which was introduced in Selberg \cite{Sel} 
together with a sophisticated consideration about the question of what is an $L$-function. 
Then we obtain an equivalent condition for an analogue of RH 
for $L$-functions in the Selberg class as in Section \ref{section_2}. 
This is the goal of this paper. 
\medskip

Before concluding the introduction, we comment on the Hilbert-P{\'o}lya conjecture, 
a conjectural possible approach to RH. 
It claims that the imaginary parts of the nontrivial zeros of $\zeta(s)$ 
are eigenvalues of some unbounded self-adjoint operator $\mathsf{D}$ acting on a Hilbert space $\mathcal{H}$. 
The Montgomery-Odlyzko conjecture on the vertical distribution of the nontrivial zeros of $\zeta(s)$ 
and the resemblance between the Weil explicit formula and the Selberg trace formula 
are strong evidence to the Hilbert-P{\'o}lya conjecture. 
No pair $(\mathcal{H},\mathsf{D})$ of space and operator had been found, 
although the conjectural pair were suggested by several authors. 
Among them, the idea of Connes~\cite{Connes99} for the conjectural Hilbert-P{\'o}lya pair $(\mathcal{H},\mathsf{D})$ 
is very attractive in the sense that 
it stands on the local--global principle in number theory (adeles and ideles), 
it enables us to understand the Weil explicit formula as a trace formula,  
and it is stated not only for $\zeta(s)$ but also Dedekind zeta-functions and Hecke $L$-functions. 
In addition, his idea is compatible with the Berry--Keating model~\cite{BK99} 
which  is an attempt to explain RH by using a physical model. 

If $E_\xi^{\omega,\nu} \in \mathbb{HB}$, we can construct a family of pairs 
$\{(\mathcal{H}^{\omega,\nu(\omega)},\mathsf{D}^{\omega,\nu(\omega)})\}_{\omega>0}$ 
of Hilbert spaces and self-adjoint operators. 
The family may be regarded as a possible realization 
of Connes' Hilbert-P{\'o}lya pair $(\mathcal{H},\mathsf{D})$ 
by allowing the perturbation parameter $\omega$. 
This topic will be treated more precisely in Section \ref{section_8}.  
\medskip

The paper is organized as follows. 
In Section \ref{section_2}, we state the main results 
Theorem \ref{thm_1}, Theorem \ref{thm_2}, Theorem \ref{thm_3} and Theorem \ref{thm_4} 
after a small preparation of notation. 
The first two theorems 
associate the above inverse problem of canonical systems to GRH. 
The third theorem is related to the necessity of GRH 
in the equivalent condition in the fourth theorem. 
The fourth theorem is the goal of this paper which is an equivalent condition for GRH 
in terms of the inverse problem of canonical systems. 
In Section \ref{section_3}, 
we briefly review the method of \cite{Su19_1} 
for the inverse problem of canonical systems to use it in the proof of the main results.
In Section \ref{section_4}, 
we prove that \cite[Theorem 1.1]{Su19_1} applies to $E_\xi^{\omega,\nu}$ 
and its generalization $E_L^{\omega,\nu}$ to $L$-functions 
in the Selberg class. 
The proof of Theorem \ref{thm_1}  is a part of this process. 
Then, we obtain Theorem \ref{thm_2} by applying \cite[Theorem 1.1]{Su19_1}
to $E_L^{\omega,\nu}$. 
In Section \ref{section_6}, we prove Theorem \ref{thm_3} 
by using the theory of de Branges spaces. 
In Section \ref{section_7}, we prove Theorem \ref{thm_4} by combining several results in former sections. 
In Section \ref{section_8}, we comment on the Hilbert-P{\'o}lya conjecture and Connes' approach about it from the viewpoint of the theory of canonical systems. 
In Section \ref{section_10}, we give miscellaneous remarks on the results and contents of this paper.

Basically, we attempt as much as possible to prove the main results in Section \ref{section_2} 
by applying general results to $L$-functions in the Selberg class 
for the convenience of applications to other class of $L$-functions. 
\medskip

\noindent
{\bf Acknowledgments}~
This work was supported by JSPS KAKENHI Grant Number JP25800007 and JP17K05163.  
Also, this work was partially supported 
by French-Japanese Projects ``Zeta Functions of Several Variables and  Applications''
in Japan-France Research Cooperative Program supported by JSPS and CNRS. 

%
%
\section{Main Results} \label{section_2}
%
%

%
%
\subsection{Selberg class and GRH}      
%
%
Let $s=\sigma+it$ be the complex variable.  
The Selberg class $\mathcal{S}$ consists of the Dirichlet series 
\begin{equation} \label{s201}
L(s) = \sum_{n=1}^{\infty} \frac{a_L(n)}{n^s}
\end{equation}
satisfying the following five axioms\footnote{
As a matter of fact, (S1) is unnecessary to define $\mathcal{S}$ 
because it is derived from (S4). 
But we put it into the axiom of $\mathcal{S}$ according to other literature. 
A positive reason is that it is convenient to define the extended Selberg class 
which is the class of Dirichlet series satisfying (S1)$\sim$(S3). 
}: 
\begin{enumerate}
\item[(S1)] The Dirichlet series \eqref{s201} converges absolutely if $\sigma>1$. 
\item[(S2)] {\bf Analytic continuation} -- There exists an integer $m \geq 0$ 
such that $(s - 1)^mL(s)$ extends to an entire function of finite order.
\item[(S3)] {\bf Functional equation} -- $L$ satisfies the functional equation
\begin{equation*}
\Lambda_L(s)=\epsilon_L \Lambda_L^\sharp(1-s),
\end{equation*}
where
\begin{equation*}
\Lambda_L(s) = Q^s \prod_{j=1}^r \Gamma(\lambda_j s + \mu_j) \cdot L(s) = \gamma_L(s)\cdot L(s), 
\end{equation*}
$\Gamma$ is the gamma function 
and $r \geq 0$, $Q > 0$, $\lambda_j > 0$, $\mu_j \in \C$ with $\Re(\mu_j)\geq 0$, 
$\epsilon_L \in \C$ with  $|\epsilon_L| = 1$ are parameters depending on $L$. 
\item[(S4)] {\bf Ramanujan conjecture} -- For every $\varepsilon > 0$, $a_L(n) \ll_\varepsilon n^\varepsilon$. 
\item[(S5)] {\bf Euler product} -- For every sufficiently large $\sigma$,  
\[
\log L(s) = \sum_{n=1}^{\infty} \frac{b_L(n)}{n^s},  
\]
where $b_L(n) = 0$ unless $n = p^m$ with $m \geq 1$, and $b_L(n) \ll n^\theta$ for some $\theta < 1/2$.
\end{enumerate}
The Riemann zeta-function and Dirichlet $L$-functions 
associated with primitive Dirichlet characters are typical members of the Selberg class $\mathcal{S}$. 
As with these examples, 
it is conjectured that all major zeta- and $L$-functions appearing in number theory, such as automorphic $L$-functions and Artin $L$-functions, 
are members of $\mathcal{S}$. 
Considering this conjecture, 
$\mathcal{S}$ is a proper class of $L$-functions 
in studying the analogue of RH for number theoretic $L$-functions. 
See the survey article~\cite{Perelli2005} of A. Perelli 
for an overview of results, conjectures, and problems relating to the Selberg class.  
\medskip

We define the degree $d_L$ of $L \in \mathcal{S}$ by 
$d_L = 2 \sum_{j=1}^{r} \lambda_j$, 
where $\lambda_j$ are numbers in (S3). 
From (S5), we have $a_L(1) = 1$ and find that 
coefficients $a_L(n)$ define a multiplicative arithmetic function. 
From (S3) and (S5), 
$L \in \mathcal{S}$ has no zeros outside the critical strip $0 \leq \sigma \leq 1$ 
except for zeros in the half-plane $\sigma \leq 0$ located at poles of the involved gamma factors. 
The zeros lie in the critical strip are called the nontrivial zeros. 
The nontrivial zeros are infinitely many unless $L \equiv 1$ 
and coincide with the zeros of the entire function\footnote{
The quantity $2 \sum_{j=1}^r(\mu_j - \frac{1}{2})$ 
is usually referred to as $\xi$-invariant and is often written as $\xi_L$ in the theory of the Selberg class, 
but we do not use the letter $\xi$ for the $\xi$-invariant of $L \in \mathcal{S}$ to avoid confusion.
}
\begin{equation} \label{s202}
\xi_L(s) = s^{m_L}(s-1)^{m_L} \Lambda_L(s)
\end{equation}
of order one, where $m_L$ is the minimal nonnegative integer $m$ in (S2).  
It is conjectured that the analogue of RH holds for all $L$-functions in $\mathcal{S}$: 
\medskip

\noindent
{\bf Grand Riemann Hypothesis (GRH).} For $L \in \mathcal{S}$, $\xi_L(s)\not=0$ unless $\Re(s)=1/2$. 
\medskip

\noindent
For $L \in \mathcal{S}$, we abbreviate to ${\rm GRH}(L)$ 
the assertion that all zeros of $\xi_L(s)$ lie on the critical line $\Re(s)=1/2$. 
\bigskip

The main subject of this paper is 
the study of ${\rm GRH}(L)$ for $L$-functions in the subclass $\mathcal{S}_{\R}$ of $\mathcal{S}$ defined by 
\begin{equation*}
\mathcal{S}_{\R} = \{ L \in \mathcal{S} \,|\, L \not\equiv 1,~L(\R) \subset \R,~\Lambda_L(\R) \subset \R \}. 
\end{equation*}
If $L \in \mathcal{S}_\R$, the gamma factor $\gamma_L$ in (S3) is not an exponential function, 
$\xi_L$ are non-constant entire functions satisfying functional equations
\begin{equation} \label{s203}
\xi_L(s)=\epsilon_L\,\xi_L(1-s) \quad \text{and} \quad \xi_L(s)=\xi_L^\sharp(s), 
\end{equation}
where the root number $\epsilon_L$ must be $1$ or $-1$. 
The second equality of \eqref{s203} means that 
$\xi_L$ is a real entire function, 
which means an entire function $F$ satisfying $F=F^\sharp$.

%
%
\subsection{Auxiliary functions for GRH} 
%
%
Let $z=u+iv$ be the complex variables relating to the variable $s$ by $s=1/2-iz$. 
To work on ${\rm GRH}(L)$ for $L \in \mathcal{S}_{\R}$, 
we introduce the two parameter family of entire functions
\begin{equation} \label{s204}
E_{L}^{\omega,\nu}(z) := \xi_L(\tfrac{1}{2}+\omega-iz)^\nu = \xi_L(s+\omega)^\nu 
\end{equation}
parametrized by $\omega \in \R_{>0}$ and $\nu \in \Z_{>0}$. 
Then, 
\begin{equation*}
(E_{L}^{\omega,\nu})^\sharp(z) = E_{L}^{\omega,\nu}(-z)
\end{equation*}
by the second equation in \eqref{s203}, and then 
$(E_{L}^{\omega,\nu})^\sharp(z) = \epsilon_L^\nu \, \xi_L(\tfrac{1}{2}-\omega-iz)^\nu$ 
by the first equation in \eqref{s203}. 
Therefore, real entire functions $A_{L}^{\omega,\nu}$ and $B_{L}^{\omega,\nu}$ 
defined by \eqref{s104} for $E=E_{L}^{\omega,\nu}$ 
are even and odd, respectively. 
The following proposition is trivial 
from definition \eqref{s204} and functional equations \eqref{s203}. 

\begin{proposition} \label{prop_2_1}
For $L \in \mathcal{S}_\R$,  
${\rm GRH}(L)$ holds if $E_L^{\omega,1} \in \overline{\mathbb{HB}}$ for all $\omega>0$. 
\end{proposition}

As easily found from \eqref{s204} and definition of $\overline{\mathbb{HB}}$, 
for fixed $\omega>0$, 
if $E_L^{\omega,\nu} \in \overline{\mathbb{HB}}$ for some $\nu \in \Z_{>0}$, 
then $E_L^{\omega,\nu} \in \overline{\mathbb{HB}}$ for arbitrary $\nu \in \Z_{>0}$. 
Therefore, $E_L^{\omega,1}$ in Proposition \ref{prop_2_1} can be replaced by $E_L^{\omega,\nu_\omega}$ 
defined for positive integers $\nu_\omega$ indexed by $\omega>0$.  
On the other hand, 
$\overline{\mathbb{HB}}$ in the statement 
can be replaced by $\mathbb{HB}$ 
without the changing of the meaning of the statement, 
although $E_L^{\omega,1} \in \overline{\mathbb{HB}}$ 
is different from $E_L^{\omega,1} \in \mathbb{HB}$ for individual $\omega>0$. 

As the converse of Proposition \ref{prop_2_1}, we have the following.

\begin{proposition} \label{prop_2_2}
Let $L \in \mathcal{S}_\R$ and $\nu \in \Z_{>0}$. 
Then, $E_{L}^{\omega,\nu}$ belongs to $\mathbb{HB}$ 
for every $\omega>1/2$ unconditionally and for every $0<\omega \leq 1/2$ under ${\rm GRH}(L)$. 
\end{proposition}
\begin{proof} 
First, we suppose that $\omega>1/2$. 
Then $E_L^{\omega,\nu}$ has no real zeros, 
since all zeros of $\xi_L(s)$ lie in the vertical strip $0 \leq \sigma \leq 1$. 
On the other hand, we find that $E_L^{\omega,\nu}$ satisfies \eqref{s103} by applying \cite[Theorem 4]{LS} to $\xi_L(s)^\nu$, 
since $\xi_L(s)$ satisfies \eqref{s203} and has only zeros in the strip $0 \leq \sigma \leq 1$. 
Hence $E_L^{\omega,\nu} \in \mathbb{HB}$. 
The case of $0< \omega  \leq 1/2$ is proved similarly under ${\rm GRH}(L)$. 
\end{proof}

The value $\omega=1/2$ in Proposition \ref{prop_2_2} comes from the trivial zero-free region of $L \in \mathcal{S}$. 
More precise relation between the zero-free region of $L$ 
and the property $E_{L}^{\omega,\nu} \in \mathbb{HB}$ 
will be discussed in Proposition \ref{prop_4_3}. 
\medskip

From Propositions \ref{prop_2_1} and \ref{prop_2_2}, 
the validity of ${\rm GRH}(L)$ for $L \in \mathcal{S}_{\R}$ is equivalent to  
the existence of a section $\R_{>0} \to \R_{>0} \times \Z_{>0};\,\omega \mapsto (\omega,\nu)$ such that 
$E_L^{\omega,\nu} \in \overline{\mathbb{HB}}$ for every $\omega>0$. 
Hence ${\rm GRH}(L)$ will be established 
if the following three steps are completed for every point of a section: 
\begin{enumerate}
\item[(3-1)] Constructing the Hamiltonian $H_L^{\omega,\nu}$ on $I=[t_0,t_1)$ from $E_L^{\omega, \nu}$ without ${\rm GRH}(L)$; 
\item[(3-2)] Constructing the solution ${}^{\rm t}[A_L^{\omega,\nu}(t,z),B_L^{\omega,\nu}(t,z)]$ 
of the canonical system on $I$ associated with $H_L^{\omega,\nu}$ 
satisfying $E_L^{\omega,\nu}(z)=A_L^{\omega,\nu}(t_0,z)-iB_L^{\omega,\nu}(t_0,z)$; 
\item[(3-3)] Showing that $\lim_{t \to t_1}J_L^{\omega,\nu}(t;z,z)=0$ for all $z \in \C_+$, 
where $J_L^{\omega,\nu}(t;z,w)$ is the function defined by \eqref{s106} 
for ${}^{\rm t}[A_L^{\omega,\nu}(t,z),B_L^{\omega,\nu}(t,z)]$. 
\end{enumerate} 
The main results stated below concern each step in (3-1)$\sim$(3-3). 

%
%
\subsection{Results on (3-1)}
%
%

We define the function $K_L^{\omega,\nu}$ on the real line by 
\begin{equation} \label{s205}
\Theta_L^{\omega,\nu}(z) 
:= 
\frac{(E_L^{\omega,\nu})^\sharp(z)}{E_L^{\omega,\nu}(z)}
= \frac{E_L^{\omega,\nu}(-z)}{E_L^{\omega,\nu}(z)} 
= \epsilon_L^\nu \, \frac{\xi_L(\tfrac{1}{2}-\omega-iz)^\nu}{\xi_L(\tfrac{1}{2}+\omega-iz)^\nu}
\end{equation}
and 
\begin{equation} \label{s206}
K_L^{\omega,\nu}(x)
:=
\frac{1}{2\pi}\int_{-\infty}^{\infty}  \Theta_L^{\omega,\nu}(u+iv)\, e^{-ix(u+iv)} \, du
\end{equation}
for $L \in \mathcal{S}_\R$ and $(\omega,\nu) \in \R_{>0}\times\Z_{>0}$.  
The integral on the right-hand side of \eqref{s206} converges absolutely if $v>0$ is sufficiently large, 
and we have 
\begin{equation} \label{s207}
K_L^{\omega,\nu}(x) 
= \epsilon_L^{\nu} \sum_{n=1}^{\lfloor \exp(x) \rfloor} 
\frac{q_L^{\omega,\nu}(n)}{\sqrt{n}} \, 
G_L^{\omega, \nu}(x-\log n)
\end{equation}
for $x > 0$ and $K_L^{\omega,\nu}(x) = 0$ for $x<0$,  
where $\epsilon_L$ is the root number in \eqref{s203}, 
$q_L^{\omega,\nu}(n)$ is an arithmetic function determined by the Dirichlet coefficients $a_L(n)$ (non-archimedean information), 
and $G_L^{\omega,\nu}$ is a certain explicit real-valued function having support in $[0,\infty)$ 
determined by the gamma factor $\gamma_L$ (archimedean information). 

We find that $K_L^{\omega,\nu}$ is a continuous function 
if $(\omega,\nu) \in \R_{>0} \times \Z_{>0}$ satisfies 
\begin{equation} \label{s208}
\nu \omega d_L >1, 
\end{equation} 
where $d_L$ is the degree of $L \in \mathcal{S}_{\R}$ (Proposition \ref{prop_4_1}). 
This condition for $(\omega,\nu)$ is technical but essential to the following construction of $H_L^{\omega,\nu}$.  
If $(\omega,\nu) \in \R_{>0} \times \Z_{>0}$ satisfies \eqref{s208}, 
\begin{equation} \label{s209}
{\mathsf K}_L^{\omega,\nu}[t]:\,f(x) \mapsto 
{\mathbf 1}_{(-\infty,t)}(x)
\, \int_{-\infty}^{t} K_L^{\omega,\nu}(x+y) \, f(y) \, dy
\end{equation} 
defines a bounded operator on $L^2(-\infty,t)$ for every $t \in \R$, 
where ${\mathbf 1}_A$ is the characteristic function of a set $A$.  
The study of ${\mathsf K}_L^{\omega,\nu}[t]$ yields a canonical system as follows. 
\medskip

\begin{theorem} \label{thm_1}
Let $L \in \mathcal{S}_{\R}$ 
and $(\omega,\nu) \in \R_{>0} \times \Z_{>0}$ be such that \eqref{s208} holds. 
Then, for every $t > 0$, \eqref{s209} defines a Hilbert-Schmidt type self-adjoint operator on $L^2(-\infty,t)$ 
having a continuous kernel,   
and ${\mathsf K}_L^{\omega,\nu}[t]=0$ for $t \leq 0$. 
Moreover, there exists $\tau=\tau(L;\omega,\nu)>0$ such that  
both $\pm 1$ are not the eigenvalues of $\mathsf{K}_L^{\omega,\nu}[t]$ 
for every $t \in [0,\tau)$. In particular, 
the Fredholm determinant 
$\det(1 \pm \mathsf{K}_L^{\omega, \nu}[t])$ 
does not vanish for every $t \in [0,\tau)$. 
\end{theorem}

The first half will be proved by applying the argument of \cite[Section 1]{Su19_1}
to $K_L^{\omega,\nu}$ under Proposition \ref{prop_4_1}. 
The latter half is Proposition \ref{prop_4_2}. 
%
By \eqref{s209}, to understand the operator 
$\mathsf{K}_L^{\omega,\nu}[t]$, 
we need only the values of $K_L^{\omega,\nu}$ on $[0,2t)$ 
that are determined by the information on the gamma factor $\gamma_L$ and 
finitely many coefficient $q_L^{\omega,\nu}(n)$'s by \eqref{s207}. 
\medskip

By Theorem \ref{thm_1}, 
the Hamiltonian $H_L^{\omega,\nu}$ on $[0,\tau)$ is defined by 
\begin{equation} \label{s210}
H_L^{\omega,\nu}(t):=
\begin{bmatrix}
1/\gamma_L^{\omega,\nu}(t) &  0 \\ 
0 & \gamma_L^{\omega,\nu}(t)
\end{bmatrix}, \quad 
\gamma_L^{\omega,\nu}(t):= \left(
\frac{\det(1+{\mathsf K}_L^{\omega,\nu}[t])}{\det(1-{\mathsf K}_L^{\omega,\nu}[t])} 
\right)^2,
\end{equation}
and it defines the canonical system
\begin{equation} \label{s211}
-\frac{d}{dt}
\begin{bmatrix}
A(t,z) \\ B(t,z)
\end{bmatrix}
= z 
\begin{bmatrix}
0 & -1 \\ 1 & 0
\end{bmatrix}
H_L^{\omega,\nu}(t)
\begin{bmatrix}
A(t,z) \\ B(t,z)
\end{bmatrix}, \quad z \in \C  
\end{equation}
on $[0,\tau)$. 
Next, we construct a unique solution of the canonical system recovering 
the entire function $E_L^{\omega,\nu}$. 

%
%
\subsection{Results on (3-2)}
%
%

Let $\tau$ be the number in Theorem \ref{thm_1}. 
Then, two integral equations 
\begin{equation*}
(1 \pm \mathsf{K}_L^{\omega,\nu}[t])\varphi_t^{\pm}(x)
=
\mathbf{1}_{(-\infty,t)}(x) K_L^{\omega,\nu}(x+t)
\end{equation*}
for unknown functions $\varphi_t^{\pm} \in L^2(-\infty,t)$  
have unique solutions 
for every $t \in [0,\tau)$, 
since $1 \pm \mathsf{K}_L^{\omega,\nu}[t]$ are invertible. 
We extend the solutions $\varphi_t^\pm$ to continuous functions on $\R$ by 
\begin{equation*}
\phi^{\pm}(t,x)  = K_L^{\omega,\nu}(x+t) \mp \int_{-\infty}^{t} K_L^{\omega,\nu}(x+y) \varphi_t^\pm(y) \, dy.
\end{equation*}
Then, we find that $|\phi^\pm(t,x)| \ll e^{c|x|}$ for some $c>0$ which may depend on $t$ 
(Proposition \ref{prop_4_1} and \cite[Lemma 2.7]{Su19_1}).  
Therefore, the two functions defined by 
\begin{equation} \label{s212}
\aligned
A_L^{\omega,\nu}(t,z) 
&:= \frac{1}{2}\,\gamma_L^{\omega,\nu}(t)^{1/2}E_L^{\omega,\nu}(z) \left( e^{izt} + \int_{t}^{\infty} \phi^+(t,x)e^{izx} \, dx \right), \\
-iB_L^{\omega,\nu}(t,z) 
&:= \frac{1}{2}\,\gamma_L^{\omega,\nu}(t)^{-1/2}E_L^{\omega,\nu}(z) \left( e^{izt} - \int_{t}^{\infty} \phi^-(t,x)e^{izx} \, dx \right)
\endaligned
\end{equation}
are analytic in the upper half-plane $\Im(z)>c'$ 
for $t \in [0,\tau)$.  

\begin{theorem} \label{thm_2}
Let $L \in \mathcal{S}_{\R}$ and 
let $(\omega, \nu) \in \R_{>0} \times \Z_{>0}$ be such that \eqref{s208} holds. 
We define $E_L^{\omega,\nu}$ by \eqref{s204}, and then define $A_{L}^{\omega,\nu}$ and $B_{L}^{\omega,\nu}$ 
by \eqref{s104} for $E=E_{L}^{\omega,\nu}$. 
Let $\tau=\tau(L;\omega,\nu)$ be the positive real number in Theorem \ref{thm_1}. 
Let $A_L^{\omega,\nu}(t,z)$ and $B_L^{\omega,\nu}(t,z)$ be the families of functions defined by \eqref{s212}. 
Then,  
\begin{enumerate}
\item $A_L^{\omega,\nu}(t,z)$ and $B_L^{\omega,\nu}(t,z)$ 
extend to real entire functions as a function of $z$ for every $t \in [0,\tau)$, 
\item $A_L^{\omega,\nu}(t,z)$ is even and $B_L^{\omega,\nu}(t,z)$ is odd as a function of $z$ for every $t \in [0,\tau)$, 
\item $A_L^{\omega,\nu}(t,z)$ and $B_L^{\omega,\nu}(t,z)$ are continuous and 
piecewise continuously differentiable functions as functions of $t$ for every $z \in \C$,  
\item ${}^{\rm t}[A_L^{\omega,\nu}(t,z), B_L^{\omega,\nu}(t,z)]$ 
solves the canonical system \eqref{s211} on $[0,\tau)$ 
associated with the Hamiltonian $H_L^{\omega,\nu}$,
\item $A_L^{\omega,\nu}(0,z)=A_L^{\omega,\nu}(z)$, $B_L^{\omega,\nu}(0,z)=B_L^{\omega,\nu}(z)$, and thus  
\begin{equation} \label{s213}
E_L^{\omega,\nu}(z)=A_L^{\omega,\nu}(0,z) - i B_L^{\omega,\nu}(0,z).  
\end{equation}
\end{enumerate} 
\end{theorem}

Theorem \ref{thm_2} 
is a generalization of \cite[Theorem 2.3]{Su12} 
which only deal with the case of $L=\zeta$ and $\nu=1$. 
The condition \eqref{s208} is $\omega>1$ 
for $L=\zeta$ and $\nu=1$ by $d_\zeta=1$. 

Theorem \ref{thm_2} will be proved by applying \cite[Theorem 1.1]{Su19_1} to 
the entire function $E_L^{\omega,\nu}$ 
under Propositions \ref{prop_4_1} and \ref{prop_4_2} 
that are proved in Section \ref{section_4}. 

%
%
\subsection{Results on (3-3)}
%
%

By the above results, we obtain the Hamiltonian $H_L^{\omega,\nu}$ defined on $[0,\tau)$ and   
the solution ${}^{\rm t}[A_L^{\omega,\nu}(t,z), B_L^{\omega,\nu}(t,z)]$ of the canonical system associated with $H_L^{\omega,\nu}$
satisfying \eqref{s213} 
{\it without} ${\rm GRH}(L)$ for $L \in \mathcal{S}_{\R}$. 
On the other hand, \cite[Theorem 40]{MR0229011} and 
\cite[Theorem 2.3]{Su12} 
suggest that their domain extends to $[0,\infty)$ under ${\rm GRH}(L)$. 
Recall that  ${\rm GRH}(L)$ implies $E_L^{\omega,\nu} \in \mathbb{HB}$ for every $(\omega, \nu)$ 
by Proposition \ref{prop_2_2}. 
If we suppose the weaker condition $E_L^{\omega,\nu} \in \overline{\mathbb{HB}}$, 
the Hamiltonian $H_L^{\omega,\nu}$ of \eqref{s210} 
and 
the solution ${}^{\rm t}[A_L^{\omega,\nu}(t,z), B_L^{\omega,\nu}(t,z)]$ of \eqref{s212} 
extend to $I=[0,\infty)$, 
because the Fredholm determinant $\det(1 \pm \mathsf{K}_L^{\omega,\nu}[t])$ 
does not vanish for every $t \geq 0$ (Proposition \ref{prop_4_4}). 
The following result asserts that 
the condition in (3-3) with $t_1=\infty$ 
is necessary for ${\rm GRH}(L)$. 

\begin{theorem} \label{thm_3}
Let $L \in \mathcal{S}_{\R}$ and let $(\omega, \nu) \in \R_{>0} \times \Z_{>0}$ 
be such that \eqref{s208} holds. 
Assume that $E_L^{\omega,\nu} \in \overline{\mathbb{HB}}$ 
if $0<\omega \leq 1/2$ and define 
$J_L^{\omega,\nu}(t;z,w)$ 
by \eqref{s106} using $A_L^{\omega,\nu}(t,z)$ and $B_L^{\omega,\nu}(t,z)$ 
of Theorem \ref{thm_1} 
for $t \geq 0$. 
Then $J_L^{\omega,\nu}(t;z,z) \not\equiv 0$ as a function of $z$ for any $t \geq 0$ and 
$\displaystyle{\lim_{t \to \infty}J_L^{\omega,\nu}(t;z,w) = 0}$ 
for every fixed $z,w \in \C$.
\end{theorem}

%
%
\subsection{Equivalent condition for GRH}
%
%

The condition in (3-3) with $t_1=\infty$ 
is a sufficient condition for $E_L^{\omega,\nu} \in \overline{\mathbb{HB}}$ 
by \cite[Proposition 2.4]{Su19_1}. 
Therefore, $E_L^{\omega,\nu} \in \overline{\mathbb{HB}}$ is equivalent that 
$H_{L}^{\omega,\nu}$ is extended to a Hamiltonian on $[0,\infty)$ 
and $\lim_{t \to \infty}J_L^{\omega,\nu}(t;z,w)=0$ if $\nu\omega d_L>1$. 
Hence we obtain the following equivalent condition for ${\rm GRH}(L)$ for $L \in \mathcal{S}_{\R}$ 
by noting that the range of $\omega$ in Proposition \ref{prop_2_1} 
can be relaxed to a decreasing sequence tending to zero. 
%
\comment{
Considering Theorem \ref{thm_2}, Theorem \ref{thm_3} and \cite[Proposition 2.4]{Su19_1}, 
$E_L^{\omega,\nu} \in \overline{\mathbb{HB}}$ is equivalent to the condition that 
$H_{L}^{\omega,\nu}$ is extended to a Hamiltonian on $[0,\infty)$ 
and $\lim_{t \to \infty}J_L^{\omega,\nu}(t;z,w)=0$ if $\nu\omega d_L>1$. 
Hence we obtain the following equivalent condition for ${\rm GRH}(L)$ for $L \in \mathcal{S}_{\R}$ 
by noting that the range of $\omega$ in Proposition \ref{prop_2_1} 
can be relaxed to a decreasing sequence tending to $0$. 
}
%
\begin{theorem} \label{thm_4}
The validity of ${\rm GRH}(L)$ for $L \in \mathcal{S}_{\R}$ is equivalent to the condition that 
there exists a sequence $(\omega_n,\nu_n) \in \R_{>0} \times \Z_{>0}$, $n \geq 1$, such that  
\begin{enumerate}
\item $\omega_m < \omega_n$ if $m>n$ and $\omega_n \to 0$ as $n \to \infty$, 
\item $\nu_n \omega_n d_L>1$, 
\item $\det(1 \pm \mathsf{K}_{L}^{\omega_n,\nu_n}[t])\not=0$ for every $t \geq 0$, 
and 
\item $\displaystyle{\lim_{t \to \infty}J_{L}^{\omega_n,\nu_n}(t;z,z)=0}$ for every $z \in \C_+$.  
\end{enumerate}
\end{theorem}

\noindent
The detailed proof of Theorem \ref{thm_4} will be given in Section \ref{section_7}. 
Also, a variant of Theorem \ref{thm_4} is stated in Section \ref{section_9_1}. 

%
\section{Inverse problem for some special canonical system} \label{section_3}
%
%

We review the way of construction of Hamiltonians in \cite{Su19_1}. 
As usual, we denote by 
\begin{equation*}
({\mathsf F}f)(z) 
 = \int_{-\infty}^{\infty} f(x) \, e^{ixz} \, dx, \quad 
({\mathsf F}^{-1} g)(z) 
 = \frac{1}{2\pi}\int_{-\infty}^{\infty} g(u) \, e^{-ixu} \, du
\end{equation*}
the Fourier integral and inverse Fourier integral, respectively, 
and use the Vinogradov symbol ``$\ll$'' to estimate a function 
as well as the Landau symbol ``$O$''. 

Let $E$ be an entire function satisfying the following five conditions: 
\begin{enumerate}
\item[(K1)] there exists a real-valued 
continuously differentiable function $\varrho$ on $\R$ 
such that $|\varrho(x)| \ll e^{-n|x|}$ for any $n>0$ 
and $E(z)=(\mathsf{F}\varrho)(z)$ for all $z \in \C$;  
\item[(K2)] there exists a real-valued continuous function $K$ on $\R$ such that 
$|K(x)| \ll \exp(c|x|)$ for some $c \geq 0$ and $E^\sharp(z)/E(z)=(\mathsf{F}K)(z)$ holds for $\Im(z) > c$; 
\item[(K3)] $K$ vanishes on $(-\infty,0)$;  
\item[(K4)] $K$ is continuously differentiable outside a discrete subset $\Lambda \subset \R$ and $|K'|$ is locally integrable on $\R$;  
\item[(K5)] there exists $0<\tau \leq \infty$ such that both $\pm 1$ are not eigenvalues of 
the operator 
$
\mathsf{K}[t]: f(x) \mapsto
\mathsf{1}_{(-\infty,t)}(x)\, \int_{-\infty}^{t} K(x+y) \, f(y) \, dy
$
 on $L^2(-\infty,t)$ for every $0 \leq t < \tau$. 
\end{enumerate}
Then, the matrix-valued function $H$ on $[0,\tau)$ defined by 
\begin{equation*} 
H(t):= \begin{bmatrix} 1/\gamma(t) & 0 \\ 0 & \gamma(t) \end{bmatrix}, 
\quad \gamma(t)=\left( \frac{\det(1+\mathsf{K}[t])}{\det(1-\mathsf{K}[t])} \right)^2 
\end{equation*}
is a Hamiltonian on $[0,\tau)$. 
We can also construct a unique solution ${}^{\rm t}[A(t,z),B(t,z)]$ 
of the canonical system associated with this $H$ 
satisfying $(A(0,z),B(0,z))=(A(z),B(z))$ 
as in \cite[Theorem 1.1]{Su19_1} 
by a way similar to \cite{Burnol2011} and \cite{Su12}, 
where $A$ and $B$ are entire functions defined by \eqref{s104}.

\section{Proof of Theorems \ref{thm_1} and \ref{thm_2} } \label{section_4}
%
%

In this section, we prove that $E_L^{\omega,\nu}$ of \eqref{s204} satisfies (K1)$\sim$(K5) in Section \ref{section_3}
(Propositions \ref{prop_4_1} and \ref{prop_4_2}). 
Then, we obtain Theorem \ref{thm_1} 
as a consequence of  Propositions \ref{prop_4_1} and \ref{prop_4_2}. 
Also, we obtain Theorem \ref{thm_2} as a consequence of \cite[Theorem 1.1]{Su19_1}. 

%
%
\subsection{Analytic properties of $E_L^{\omega,\nu}$ and $\Theta_L^{\omega,\nu}$} \label{subsec_3_1}
%
%

\begin{lemma} \label{lem_4_1}
The entire function $E_L^{\omega,\nu}$ define by \eqref{s204} 
for $L \in \mathcal{S}_\R$ and $(\omega,\nu) \in \R_{>0} \times \Z_{>0}$ 
satisfies (K1). 
More precisely, if we define 
\begin{equation} \label{s401}
\varrho_L^{\omega, \nu}(x) = \frac{1}{2\pi} \int_{\Im(z)=c} E_L^{\omega,\nu}(z) e^{-izx} dz 
\end{equation}
by taking some $c \in \R$, then it is a real-valued function in $C^\infty(\R)$ satisfying 
the estimate 
$|\varrho_L^{\omega, \nu}(x)| \ll_n e^{-n|x|}$ for every $n \in \N$ such that 
\begin{equation*}
E_L^{\omega,\nu}(z) 
=(\mathsf{F}\varrho_L^{\omega,\nu})(z)=\int_{-\infty}^{\infty} \varrho _L^{\omega, \nu}(x) e^{izx} \, dx 
\end{equation*}
holds for all $z \in \C$. 
\end{lemma}
\begin{proof}
From (S1), $L(s)$ is bounded in $\Re(s) \geq 2$. 
From (S2), (S3) and the Phragm{\'e}n--Lindel{\"o}f principle, 
$(s-1)^{m_L}L(s)$ is bounded by a polynomial of $s$ in any vertical strip of finite width. 
Thus, $\xi_L(s)$ decays faster than $|s|^{-n}$ for any $n \in \N$ in any vertical strip of finite width. 
Hence $\varrho_L^{\omega, \nu}$ is defined by \eqref{s401} independent of $c$ and belongs to $C^\infty(\R)$. 
From the second equation of \eqref{s203}, we have $(E_L^{\omega,\nu})^\sharp(z)=E_L^{\omega,\nu}(-z)$. 
Thus, by taking $c=0$ in \eqref{s401}, we find that $\varrho_L^{\omega, \nu}$ is real-valued. 
By moving the path of integration in \eqref{s401} to the upside or downside, we see that 
$|\varrho_L^{\omega, \nu}(x)| \ll_n e^{-n|x|}$ for every $n \in \N$. 
Taking the Fourier transform of \eqref{s401}, 
we obtain $E_L^{\omega,\nu}(z) =(\mathsf{F}\varrho_L^{\omega,\nu})(z)$. 
\end{proof}

\begin{lemma} \label{lem_4_2} 
Let $\Theta_L^{\omega,\nu}$ be the meromorphic function in $\C$ 
define by \eqref{s205} 
for $L \in \mathcal{S}_\R$ and $(\omega,\nu) \in \R_{>0} \times \Z_{>0}$. 
Then the estimate 
\begin{equation} \label{s402}
|\Theta_L^{\omega,\nu}(u+iv)| \ll_v (1+|u|)^{-\nu\omega d_L}
\end{equation}
holds for any fixed $v>1/2+\omega$, 
Moreover, the estimate 
\begin{equation} \label{s403}
|\Theta_L^{\omega,\nu}(u+iv)| \ll_\delta (1+v)^{-\nu\omega d_L}
\end{equation}
holds uniformly for $u \in \R$ and $v \geq 1/2+\omega+\delta$ for any fixed $\delta>0$. 
\end{lemma}
\begin{proof} 
We have $\xi_L(s)=\xi_L^{\infty}(s)L(s)$ by putting   
\begin{equation} \label{s404}
\xi_L^{\infty}(s)=s^{m_L}(s-1)^{m_L} \cdot \gamma_L(s) =s^{m_L}(s-1)^{m_L} \cdot Q^{s} \prod_{j=1}^{r}\Gamma(\lambda_j s + \mu_j). 
\end{equation}
Then, in order to prove \eqref{s402} and \eqref{s403}, it is sufficient to prove 
\begin{equation} \label{s405}
\left|
\frac{\xi_L^\infty(s-\omega)}{\xi_L^\infty(s+\omega)}
\right| \ll (1+|u|)^{-\omega d_L}
\quad 
\text{and} 
\quad 
\left|
\frac{\xi_L^\infty(s-\omega)}{\xi_L^\infty(s+\omega)}
\right| \ll (1+v)^{-\omega d_L}
\end{equation}
for $s=1/2-i(u+iv)$ with $v \geq 1/2+\omega+\delta$ by \eqref{s202} and \eqref{s205}, 
since $L(s-\omega)/L(s+\omega)$ 
is expressed as an absolutely convergent Dirichlet series 
in a right-half plane $\Re(s)>1+\omega$ and hence it is bounded there. 

Let $\lambda>0$, $\mu=\mu_0+i\mu_1 \in \C$ with $\mu_0 \geq 0$. 
Using the Stirling formula
\begin{equation*} 
\Gamma(s)=\sqrt{\frac{2\pi}{s}}\left(\frac{s}{e}\right)^s(1+O_\delta(|s|^{-1})) \quad 
|s| \geq 1, ~|\arg s| < \pi-\delta
\end{equation*}
for $\delta=\pi/4$ and $s=1/2-i(u+iv)$ ($u \in \R$, $v>0$), we have 
\begin{equation*} 
\aligned 
\left|
\frac{\Gamma(\lambda(s-\omega)+\mu)}{\Gamma(\lambda(s+\omega)+\mu)}
\right|
& = \sqrt{\left|1+\frac{2\lambda\omega}{\lambda(s-\omega)+\mu}\right|}
\cdot\frac{1+O(|\lambda(s-\omega)+\mu|^{-1})}{1+O(|\lambda(s+\omega)+\mu|^{-1})} \cdot e^{2\lambda\omega} \\
& \times 
\exp\Bigl[ - 2\lambda\omega\log\left|\lambda(\tfrac{1}{2}+v-\omega)+\mu_0-i(\lambda u-\mu_1)\right| \\
& \quad + \left(\lambda(\tfrac{1}{2}+v+\omega)+\mu_0\right) \log
\left| 
1- \frac{2\lambda\omega}{\lambda(\tfrac{1}{2}+v+\omega)+\mu_0-i(\lambda u-\mu_1)}
\right| \\
& \quad 
+ (\lambda u-\mu_1) \arg\left( 
1- \frac{2\lambda\omega}{\lambda(\tfrac{1}{2}+v+\omega)+\mu_0-i(\lambda u-\mu_1)} \right)
\Bigr]. 
\endaligned
\end{equation*}
On the right-hand side, 
\begin{equation*}
\sqrt{\left|1+\frac{2\lambda\omega}{\lambda(s-\omega)+\mu}\right|}
\cdot\frac{1+O(|\lambda(s-\omega)+\mu|^{-1})}{1+O(|\lambda(s+\omega)+\mu|^{-1})} \cdot e^{2\lambda\omega},
\end{equation*}
\begin{equation*}
\left(\lambda(\tfrac{1}{2}+v+\omega)+\mu_0\right) \log
\left| 
1- \frac{2\lambda\omega}{\lambda(\tfrac{1}{2}+v+\omega)+\mu_0-i(\lambda u-\mu_1)}
\right| 
\end{equation*}
and
\begin{equation*}
(\lambda u-\mu_1) \arg\left( 
1- \frac{2\lambda\omega}{\lambda(\tfrac{1}{2}+v+\omega)+\mu_0-i(\lambda u-\mu_1)} \right)
\end{equation*}
are uniformly bounded for $u \in \R$ and $v \geq 1/2+\omega+\delta$.  
Hence 
\begin{equation} \label{s406}
\aligned 
\left|
\frac{\Gamma(\lambda(s-\omega)+\mu)}{\Gamma(\lambda(s+\omega)+\mu)}
\right|
\ll
\exp\Bigl[ - 2\lambda\omega\log\left|\lambda(\tfrac{1}{2}+v-\omega)+\mu_1-i(\lambda u-\mu_2)\right| \Bigr],
\endaligned
\end{equation}
where the implied constant depends on $\omega$, $\lambda$, $\mu$ and $\delta >0$. 
This implies the first estimate of \eqref{s405} by \eqref{s404} and the definition of $d_L$. 
On the other hand, the right-hand side of \eqref{s406} takes the maximum 
$[\lambda(\tfrac{1}{2}+v-\omega)+\mu_1]^{-2\lambda\omega}$ 
at $u=\mu_1/\lambda$ as a function of $u$. 
Hence 
\begin{equation*}
\aligned 
\left|
\frac{\Gamma(\lambda(s-\omega)+\mu)}{\Gamma(\lambda(s+\omega)+\mu)}
\right|
\ll
(1+v)^{-2\lambda\omega}
\endaligned
\end{equation*}
holds uniformly for $u \in \R$ and $v \geq 1/2+\omega+\delta$, 
where the implied constant depends on $\omega$, $\lambda$, $\mu$ and $\delta>0$. 
This implies the second estimate of \eqref{s405} by \eqref{s404} 
and the definition of $d_L$. 
\end{proof}

Let $a_L$ be the arithmetic function defined by the Dirichlet series \eqref{s201} of $L \in \mathcal{S}_{\R}$. 
Then, by $a_L(1) = 1 \not=0$, the Dirichlet inverse $a_{L}^{-1}$ exists and is given by 
$a_{L}^{-1}(n) = - \sum_{d|n,\,d<n} a_L(n/d)a_L^{-1}(d)$ for $n>1$ and $a_L^{-1}(1)=1$. 
We have 
\begin{equation*}
L(s)^{\pm k} = \sum_{n=1}^{\infty} \frac{a_L^{\pm k}(n)}{n^s}, \quad 
a_L^{\pm k} = \underbrace{ a_L^{\pm 1} \ast \cdots \ast a_L^{\pm 1} }_{\text{$k$ times}}
\end{equation*}
for any positive integer $k$, where $a_L^1=a_L$ and $\ast$ is the Dirichlet convolution for arithmetic functions. 
Using these arithmetic functions, we define
\begin{equation} \label{s407}
q_L^{\omega,\nu}(n) := n^\omega \sum_{d|n} 
\frac{a_L^\nu(n/d) a_L^{-\nu}(d)}{d^{2\omega}} 
\end{equation}
for natural numbers $n$. 

Next, we introduce the function $g_{\omega,\lambda,\mu}$ defined on the positive real line by
\begin{equation*}
g_{\omega,\lambda,\mu}(y) 
= \frac{1}{\lambda \Gamma(2\lambda \omega)} \, 
y^{\omega - \frac{1}{2} - \frac{\mu}{\lambda}} (1- y^{-\frac{1}{\lambda}})^{2 \lambda \omega -1}
\end{equation*}
for $y>1$ and $g_{\omega,\lambda,\mu}(y)=0$ for $0<y<1$. Then, we define 
\begin{equation*}
\tilde{g}_L^{\omega} = g_{\omega,\lambda_1,\mu_1} \ast g_{\omega,\lambda_2,\mu_2} \ast \cdots \ast g_{\omega,\lambda_r,\mu_r},
\end{equation*}
\begin{equation*}
\tilde{g}_L^{\omega, \nu} = Q^{-2\nu\omega}\cdot\underbrace{\tilde{g}_L^{\omega} \ast \cdots \ast \tilde{g}_L^{\omega}}_{\text{$\nu$ times}}
\end{equation*}
by using quantities $r$, $\lambda_j$, $\mu_j$ and $Q$ in (S3), 
where $\ast$ is the multiplicative convolution $(f\ast g)(x)=\int_{0}^{\infty} f(x/y)g(y) \, \frac{dy}{y}$.  
In addition, using the partial fraction decomposition
\begin{equation*} 
\left( \frac{(s-\omega)(s-\omega-1)}{(s+\omega)(s+\omega-1)} \right)^{\nu m_L} 
= 1 + \sum_{k=1}^{\nu m_L} \left( \frac{X_k(\omega)}{(s+\omega-1)^k} + \frac{Y_k(\omega)}{(s+\omega)^k} \right), 
\end{equation*}
we define
\begin{equation*}
r_L^{\omega,\nu}(y)= \delta_{0}(y) + \sum_{k=1}^{\nu m_L} \frac{1}{(k-1)!}
\left( X_k(\omega) \, y^{1/2} + Y_{k}(\omega) \, y^{-1/2} \right) 
y^{-\omega}(\log y)^{k-1} \mathbf{1}_{[1,\infty)}(y)
\end{equation*}
for $y>1$ and $r_L^{\omega,\nu}(y)=0$ for $0<y<1$, 
where $\delta_0$ is the Dirac mass at the origin. 

We now define 
\begin{equation*}
g_L^{\omega,\nu} := r_L^{\omega,\nu} \ast \tilde{g}_L^{\omega, \nu}. 
\end{equation*}
Then $g_L^{\omega, \nu}$ is a $C^\infty$-function on $(1,\infty)$ and vanishes on $(0,1)$. 
The behavior of $g_L^{\omega, \nu}$ near $y=1$ is singular if $\nu>0$ is small, 
but if $\nu$ is sufficiently large with respect to $\omega$, 
$g_L^{\omega,\nu}$ is continuous at $y=1$. 

Finally, we define the function $K_L^{\omega,\nu}$ on the real line by 
\begin{equation} \label{s408}
\aligned
k_L^{\omega,\nu}(y) & = \epsilon_L^\nu \sum_{n=1}^{\lfloor y \rfloor} \frac{q_L^{\omega,\nu}(n)}{\sqrt{n}} \, 
g_L^{\omega,\nu}\left(\frac{y}{n}\right), \\ 
K_L^{\omega,\nu}(x) & = k_L^{\omega,\nu}(e^x)
= \epsilon_L^\nu \sum_{n=1}^{\lfloor \exp(x) \rfloor} \frac{q_L^{\omega,\nu}(n)}{\sqrt{n}} \, 
G_L^{\omega,\nu}(x-\log n)
\endaligned 
\end{equation}
for $x>0$, and $K_L^{\omega,\nu}(x)=0$ for $x<0$, 
where $G_L^{\omega,\nu}(x) = g_L^{\omega,\nu}(\exp(x))$. 
The value $K_L^{\omega,\nu}(0)$ may be undefined if $\nu$ is small with respect to $\omega>0$, 
but it is understood as $K_L^{\omega,\nu}(0)=0$ for large $\nu$,  
since $q_L^{\omega,\nu}(1)=1$ and $g_L^{\omega,\nu}(1) = 0$ if $\nu$ is large.

\begin{proposition} \label{prop_4_1} 
Define $K_L^{\omega,\nu}$ as above 
for $L \in \mathcal{S}_\R$ and $(\omega,\nu) \in \R_{>0} \times \Z_{>0}$. 
Then, 
\begin{enumerate}
\item $K_L^{\omega,\nu}$ is a continuous real-valued function on $\R \setminus \{ \log n \,|\, n \in \N \}$ 
vanishing on the negative real line $(-\infty,0)$,
\item $K_L^{\omega,\nu}$ is continuously differentiable on $\R \setminus \{ \log n \,|\, n \in \N \}$, 
\item the Fourier integral formula 
\begin{equation} \label{s409}
\Theta_L^{\omega,\nu}(z)
= (\mathsf{F}K_L^{\omega,\nu})(z) 
= \int_{0}^{\infty} K_L^{\omega,\nu}(x) \, e^{izx} \, dx 
\end{equation}
holds for $\Im(z)>1/2+\omega$ with the absolute convergence of the integral on the right-hand side. 
In particular, $K_L^{\omega,\nu}$ of \eqref{s408} coincides with the function defined in \eqref{s206} 
by taking the Fourier inversion formula of \eqref{s409}.
\end{enumerate}

Suppose that $(\omega, \nu)$ satisfies the condition \eqref{s208}. Then, 
\begin{enumerate}
\item[(4)] $K_L^{\omega,\nu}$ is a continuous function on $\R$, 
\item[(5)] $|K_L^{\omega,\nu}(x)| \ll \exp(c|x|)$ for some $c>0$, 
\item[(6)] $\vert \frac{d}{dx}K_L^{\omega,\nu}(x) \vert$ is locally integrable. 
\end{enumerate}
Hence, $E_L^{\omega,\nu}$ satisfies (K1)$\sim$(K4) under \eqref{s208} by combining with Lemma \ref{lem_4_1}. 

Moreover, if $\omega$ and $\nu$ satisfy $\nu \omega d_L > k+1$ for some $k \in \N$, 
$K_L^{\omega,\nu}$ belongs to $C^k(\R)$ 
\end{proposition}
\begin{proof} Properties (1) and (2) are trivial by definition, 
and (5) is a simple consequence of \eqref{s206} and \eqref{s402}. 
To prove (3), we use the variable $s=1/2-iz$ for convenience. 
If $\omega>0$, $\lambda>0$, $\Re(\mu) \geq 0$, we have
\begin{equation} \label{s410}
\frac{\Gamma(\lambda (s-\omega) + \mu)}{\Gamma(\lambda(s+\omega) + \mu)}
= \frac{1}{\lambda \Gamma(2\lambda \omega)}
\int_{1}^{\infty} y^{\omega - \frac{1}{2} - \frac{\mu}{\lambda}}(1-y^{-\frac{1}{\lambda}})^{2 \lambda \omega -1} \, y^{\frac{1}{2}-s} \frac{dy}{y} 
\end{equation}
for $\Re(s)>\omega - \Re(\mu)/\lambda$ by \cite[(5.35) of p.195]{Ob}. 
Therefore, we obtain 
\begin{equation} \label{s411}
\aligned
\int_{0}^{\infty} \tilde{g}_L^{\omega, \nu}(y) \cdot y^{\frac{1}{2}-s} \, \frac{dy}{y}
= 
\left( Q^{-2\omega}
\prod_{j=1}^{r}
\frac{\Gamma(\lambda_j (s-\omega) + \mu_j)}{\Gamma(\lambda_j(s+\omega) + \mu_j)}
\right)^\nu 
\endaligned
\end{equation}
for $\Re(s)>\max_{1 \leq j \leq r}(\omega-\Re(\mu_j)/\lambda_j)$ by applying \cite[Theorem 44]{Tit2} repeatedly to \eqref{s410}. 

Applying the formula 
\begin{equation*}
\frac{1}{(k-1)!}\int_{1}^{\infty} y^{-a-\frac{1}{2}}(\log y)^{k-1} \cdot y^{\frac{1}{2}-s} \frac{dy}{y} = \frac{1}{(s+a)^k}, \quad \Re(s+a)>0
\end{equation*}
to $r_L^{\omega,\nu}$, we have 
\begin{equation} \label{s412}
\int_{1}^{\infty} r_L^{\omega,\nu}(y) \cdot y^{\frac{1}{2}-s} \frac{dy}{y} 
= \left( \frac{(s-\omega)(s-\omega-1)}{(s+\omega)(s+\omega-1)} \right)^{\nu m_L}
\end{equation}
for $\Re(s)>1-\omega$ (we assumed $\omega>0$). 

Then we obtain 
\begin{equation*}
\aligned
\int_{0}^{\infty} & g_L^{\omega, \nu}(y) \cdot y^{\frac{1}{2}-s} \, \frac{dy}{y}
 =
\int_{1}^{\infty} g_L^{\omega, \nu}(y) \cdot y^{\frac{1}{2}-s} \, \frac{dy}{y} \\
= &
\left( \frac{(s-\omega)(s-\omega-1)}{(s+\omega)(s+\omega-1)} \right)^{\nu m_L}
\left( Q^{-2\omega}
\prod_{j=1}^{r}
\frac{\Gamma(\lambda_j (s-\omega) + \mu_j)}{\Gamma(\lambda_j(s+\omega) + \mu_j)}
\right)^\nu 
 = 
\left(
\frac{\xi_L^\infty(s-\omega)}{\xi_L^\infty(s+\omega)}
\right)^\nu
\endaligned
\end{equation*}
for $\Re(s)>{\rm max}(1-\omega,\max_{1 \leq j \leq r}(\omega-\Re(\mu_j)/\lambda_j))$ 
by applying \cite[Theorem 44]{Tit2} to \eqref{s411} and \eqref{s412}. 
On the other hand, we have  
\begin{equation*}
\aligned
\left( \frac{L(s-\omega)}{L(s+\omega)} \right)^\nu  
 & = \sum_{m=1}^{\infty}\frac{a_L^\nu(m)m^\omega}{m^s}\sum_{n=1}^{\infty}\frac{a_L^{-\nu}(n)n^{-\omega}}{n^s} \\
 & = \sum_{n=1}^{\infty} \frac{1}{n^s} \left( n^\omega \sum_{d|n} \frac{a_L^\nu(n/d) a_L^{-\nu}(d)}{d^{2\omega}} \right)
 = \sum_{n=1}^{\infty} \frac{q_L^{\omega,\nu}(n)}{n^s}
\endaligned
\end{equation*}
by definition \eqref{s407}, where the series converges absolutely for $\Re(s)>1+\omega$. 
By definition \eqref{s408}, we have formally 
\begin{equation*}
\aligned
\int_{0}^{\infty} k_L^{\omega,\nu}(y) \, y^{\frac{1}{2}-s} \, \frac{dy}{y} 
&= \epsilon_L^\nu\sum_{n=1}^{\infty} \frac{c_L^{\omega,\nu}(n)}{n^s} \int_{0}^{\infty} g_L^{\omega,\nu}(y/n) \, (y/n)^{\frac{1}{2}-s} \, \frac{dy}{y} \\
&= \epsilon_L^\nu\left( \frac{\xi_L^\infty(s-\omega)}{\xi_L^\infty(s+\omega)} \frac{L(s-\omega)}{L(s+\omega)} \right)^\nu,
\endaligned
\end{equation*}
and it is justified by Fubini's theorem for $\Re(s)>1+\omega$. 
By the changing of variables $y=e^x$ and $s=1/2-iz$, we obtain \eqref{s409} 
and complete the proof of (3). 

We prove (4) and the last line of Proposition \ref{prop_4_1}. 
By \eqref{s402}, $\Theta_L^{\omega,\nu}(u+iv)$ belongs to $L^1(\R)$ as a function of $u$ if $v$ is sufficiently large. 
Thus, $K_L^{\omega,\nu}$ is uniformly continuous on $\R$ by \eqref{s206}. 
Moreover, by \eqref{s206} and \eqref{s402}, the formula 
\begin{equation*}
\frac{d^n}{dx^n}K_L^{\omega,\nu}(x) = \frac{1}{2\pi} \int_{\Im(z)=c} \Theta_L^{\omega,\nu}(z) (-iz)^n \, e^{-izx} \, dz 
\end{equation*}
holds together with the absolute convergence of the integral on the right-hand side. 
Therefore, this   
shows that $K_L^{\omega,\nu}$ is $C^k$ if $\nu\omega d_L>k+1$. 

Finally, we prove (6). The derivative $\frac{d}{dx}K_L^{\omega,\nu}$ is locally integrable by (4). 
On the other hand, the set of possible singularities of $\frac{d}{dx}K_L^{\omega,\nu}$ is discrete in $\R$ by (2), 
and $\frac{d}{dx}K_L^{\omega,\nu}$ does not change its sign infinitely often around any possible singularity except for $x=+\infty$ 
by definition of $g_L^{\omega,\nu}$. Therefore, the local integrability of $\frac{d}{dx}K_L^{\omega,\nu}$ 
implies the local integrability of $|\frac{d}{dx}K_L^{\omega,\nu}|$. 
\end{proof}

By Proposition \ref{prop_4_1}, we find that $E_L^{\omega,\nu}$ satisfies (K1)$\sim$(K4) if $(\omega, \nu)$ satisfies \eqref{s208}. 
Successively, we show that $E_L^{\omega,\nu}$ satisfies (K5) for sufficiently small $\tau>0$ unconditionally. 

%
%
\subsection{Non-vanishing of Fredholm determinants: Unconditional cases} 
%
%

In this section, we understand that $\mathsf{K}_L^{\omega,\nu}f$ 
is a function defined by the integral 
\begin{equation} \label{s302}
({\mathsf K}_L^{\omega,\nu}f)(x) = \int_{-\infty}^{\infty} K_L^{\omega,\nu}(x+y) \, f(y) \, dy
\end{equation}
if the right-hand side converges absolutely and locally uniform for a function $f$, 
because we do not assume that $E_L^{\omega,\nu} \in \overline{\mathbb{HB}}$ 
(which implies that $\mathsf{K}_L^{\omega,\nu}f$ belongs to $L^2(\R)$ 
for every $f \in L^2(\R)$ 
by \cite[Lemmas 2.1 and 2.2]{Su19_1}). 

\begin{proposition} \label{prop_4_2} 
Let $L \in \mathcal{S}_\R$. 
Suppose that $(\omega, \nu)$ satisfies \eqref{s208} and define the operator $\mathsf{K}_L^{\omega,\nu}[t]$ 
on $L^2(-\infty,t)$ by \eqref{s209}. 
Then, there exists $\tau=\tau(L;\omega,\nu)>0$ such that both $\pm 1$ are not eigenvalues of $\mathsf{K}_L^{\omega,\nu}[t]$ 
for every $0 \leq t < \tau$, 
that is, both $1 \pm {\mathsf K}_L^{\omega,\nu}[t]$ are invertible operators on $L^2(-\infty,t)$ 
for every $0 \leq t < \tau$. 
Thus $E_L^{\omega,\nu}$ satisfies (K5) for $[0,\tau)$. 
\end{proposition}
\begin{proof} 
The spectrum of ${\mathsf K}_L^{\omega,\nu}[t]$ is discrete and consists of eigenvalues, 
since ${\mathsf K}_L^{\omega,\nu}[t]$ is a Hilbert-Schmidt operator on $L^2(-\infty,t)$ 
by (K2) and (K3). 
The statement of the proposition 
is equivalent that ${\mathsf K}_L^{\omega,\nu}[t] f \not= \pm f$ for any $f \in L^2(-\infty,t)$, 
because $1 - \mu{\mathsf K}_L^{\omega,\nu}[t]$ is invertible if $1/\mu$ is not an eigenvalue. 
In addition, ${\mathsf K}_L^{\omega,\nu}[t] f \not= \pm f$ is equivalent that ${\mathsf P}_t {\mathsf K}_L^{\omega,\nu} f \not= \pm f$, 
since ${\mathsf P}_t f = f$ for $f \in L^2(-\infty,t)$, 
where ${\mathsf P}_t$ is the orthogonal projection from $L^2(\R)$ to $L^2(-\infty,t)$. 

Suppose that ${\mathsf P}_t{\mathsf K}_L^{\omega,\nu}f  = \pm f $ for some $0 \not= f \in L^2(-\infty,t)$. 
We have 
\begin{equation*}
({\mathsf P}_t{\mathsf K}_L^{\omega,\nu} f)(x) 
=  {\mathbf 1}_{(-\infty,t)}(x)\int_{-\infty}^{t} K_L^{\omega,\nu}(x+y)f(y)\, dy = 0 
\end{equation*}
for $-\infty<x<-t$ and $f\in L^2(-\infty,t)$. 
Therefore, ${\mathsf P}_t{\mathsf K}_L^{\omega,\nu}f$ is a function on $\R$ 
having support in $[-t,\infty)$, 
and hence the assumption implies that $f$ has a compact support contained in $[-t,t]$. 

We put $g={\mathsf K}_L^{\omega,\nu}f$. 
Then, we have 
$({\mathsf F} g)(z)  = \Theta_L^{\omega,\nu}(z)({\mathsf F}f)(-z)$
for $\Im(z)>1/2+\omega$ by  \cite[Lemma 2.2]{Su19_1}. 
This means that 
\begin{equation*} 
({\mathsf F}g_{-v})(u)  = \Theta_L^{\omega,\nu}(u+iv)({\mathsf F}f_{v})(-u) 
\end{equation*}
for $u \in \R$ if $v>1/2+\omega$, where we put $g_{-v}(x) = g(x)e^{-xv}$ and $f_{v}(x) = f(x)e^{xv}$. 
%
%
Therefore, we have 
\begin{equation*} 
\aligned 
\Vert {\mathsf F}g_{-v} \Vert^2 
& = \Vert \Theta_L^{\omega,\nu}(\cdot+iv)({\mathsf F}f_{v})(-\cdot) \Vert^2 \\
& \leq M_v^2 \Vert {\mathsf F}f_{v} \Vert^2 = 2\pi M_v^2 \Vert f_{v} \Vert^2
=  2\pi M_v^2 \int_{-\infty}^{t} |f(x)|^2 e^{2vx} \, dx \\
&\leq 2\pi M_v^2 e^{2vt} \int_{-\infty}^{t} |f(x)|^2\, dx = 2\pi M_v^2 e^{2vt} \Vert f \Vert^2, 
\endaligned
\end{equation*}
where $\Vert \cdot \Vert = \Vert \cdot \Vert_{L^2(\R)}$ and $M_v = \max_{u \in \R}|\Theta_L^{\omega,\nu}(u+iv)|$. 
Therefore, we have  
\begin{equation} \label{s413}
\Vert g_{-v} \Vert \leq M_v e^{vt} \Vert f \Vert
\end{equation}
by $\Vert {\mathsf F}g_{-v} \Vert^2 = 2\pi \Vert g_{-v} \Vert^2$. 
On the other hand, the equality ${\mathsf P}_t{\mathsf K}_L^{\omega,\nu}f  = \pm f $ implies 
\[
\Vert {\mathsf P}_tg_{-v} \Vert^2 = \Vert f_{-v} \Vert^2 
= \int_{-\infty}^{t} |f(x)|^2 e^{-2vx} \, dx
\geq e^{-2vt} \int_{-\infty}^{t} |f(x)|^2  \, dx = e^{-2vt} \Vert f \Vert^2 
\]
for every $v>0$. Therefore, 
\begin{equation} \label{s414}
e^{vt} \Vert{\mathsf P}_tg_{-v} \Vert  \geq  \Vert f \Vert.  
\end{equation}
By \eqref{s413} and \eqref{s414}, we have
$\Vert g_{-v} \Vert \leq M_v e^{2vt} \Vert{\mathsf P}_tg_{-v} \Vert$ 
which implies 
\begin{equation*}
\int_{-t}^{\infty} |g(x)|^2 e^{-2vx} \, dx \leq M_v^2 e^{4vt}\int_{-t}^{t} |g(x)|^2 e^{-2vx} \, dx, 
\end{equation*}
since $g$ has support in $[-t,\infty)$. By \eqref{s403}, we have 
\begin{equation*}
 M_v^2 e^{4vt} \ll \exp\Bigl(4vt-2\nu\omega d_L\log (1+v)\Bigr).
\end{equation*}
Therefore, $M_v^2 e^{4vt}<1$ and thus
\begin{equation*}
\int_{-t}^{\infty} |g(x)|^2 e^{-2vx} \, dx < \int_{-t}^{t} |g(x)|^2 e^{-2vx} \, dx. 
\end{equation*}
if $t>0$ is sufficiently small with respect to fixed $v > 1/2+\omega$.  
This is a contradiction. Hence, ${\mathsf P}_t{\mathsf K}_L^{\omega,\nu}f  = \pm f$ is impossible for every sufficiently small $t>0$.  
\end{proof}

%
%
\subsection{Further analytic properties of $\Theta_L^{\omega,\nu}$} 
%
%

Let $H^\infty=H^\infty(\C_+)$ be the space of all bounded analytic functions in $\C_+$. 
A function $\Theta \in H^\infty$  
is called an {\it inner function} in $\C_+$ 
if $\lim_{y \to 0+}|\Theta(x+iy)|=1$ for almost all $x \in \R$ 
with respect to the Lebesgue measure. 
If an inner function $\Theta$ in $\C_+$ is extended to a meromorphic function in $\C$, 
it is called a {\it meromorphic inner function} in $\C_+$. 

\begin{proposition} \label{prop_4_3}
Let $L \in \mathcal{S}_\R$, $\omega_0 \geq 0$ and $\nu \in \Z_{>0}$. 
Then the following statements are equivalent: 
\begin{enumerate}
\item $L(s)\not=0$ for $\Re(s)>\frac{1}{2}+\omega_0$, 
\item $E_L^{\omega,\nu}$ belongs to $\mathbb{HB}$ for every $\omega>\omega_0$, 
\item $\Theta_L^{\omega,\nu}$ is a meromorphic inner function in $\C_+$ for every $\omega>\omega_0$. 
\end{enumerate}
The value of $\nu$ does not affect the above equivalence. 
\end{proposition}
\begin{proof}
Assume $0 \leq \omega_0 \leq 1/2$, since we have nothing to say for $\omega > 1/2$ 
by Proposition \ref{prop_2_2} and \cite[Lemma 2.1]{Su19_1}. 
We find that (1) implies that $E_L^{\omega,\nu}$ satisfies \eqref{s103} for every $\omega>\omega_0$ 
in a way similar to the proof of Proposition \ref{prop_2_2}. 
If $E_L^{\omega,\nu}(z)=\xi_L^{\omega}(\tfrac{1}{2}+\omega-iz)^\nu$ has a real zero for some $\omega>\omega_0$, 
it implies that $L(s)$ has a zero in $\Re(s)>\frac{1}{2}+\omega_0$, since
\begin{equation*}
\xi_L(\tfrac{1}{2}+\omega-iz)=\xi_L(\tfrac{1}{2}+\omega_0-i(z+i(\omega-\omega_0))). 
\end{equation*}
Thus $E_L^{\omega,\nu} \in \mathbb{HB}$ and we obtain (1)$\Rightarrow $(2). 
The implication (2)$\Rightarrow$(3) is a consequence of \cite[Lemma 2.1]{Su19_1}. 
The implication (3)$\Rightarrow$(1) is proved in a way similar to the proof of Theorem 2.3 (1) in \cite{Su11}.
\end{proof}

The value $\omega=1/2$ corresponds to the abscissa $\sigma=1$ 
of the absolute convergence of the Dirichlet series \eqref{s201}.  
The non-vanishing of $L \in \mathcal{S}$ on the line $\sigma = 1$ is an important problem 
because it relates to an analogue of the prime number theorem of $L \in \mathcal{S}$ for example. 
Conrey--Ghosh~\cite{CoGh93} proved the non-vanishing of $L \in \mathcal{S}$ on the line $\sigma = 1$ 
subject to the truth of the Selberg orthogonality conjecture. 
Kaczorowski--Perelli~\cite{KP03} obtained the non-vanishing of $L \in \mathcal{S}$ on the line $\sigma = 1$ 
under a weak form of the Selberg orthogonality conjecture. 
As mentioned before, it is conjectured that $\mathcal{S}$ consists only of automorphic $L$-functions. 
The non-vanishing for automorphic $L$-functions on the line $\sigma=1$ 
had been proved unconditionally in Jacquet--Shalika~\cite{JaSh76}.

%
%
\section{Proof of Theorem \ref{thm_3}} \label{section_6}
%
%

In this section, we prove that $E_L^{\omega,\nu}$ of \eqref{s204} satisfies (K5) for $\tau=\infty$ 
if $E_L^{\omega,\nu} \in \overline{\mathbb{HB}}$ (Proposition \ref{prop_4_4}). 
Then we obtain Theorem \ref{thm_3} as a consequence of \cite[Theorem 1.2]{Su19_1}. 

%
%
\subsection{Non-vanishing of Fredholm determinants: Conditional cases} 
%
%

We suppose that $E_L^{\omega,\nu} \in \overline{\mathbb{HB}}$ 
throughout this subsection, otherwise it will be mentioned.  
Then, $\Theta_L^{\omega,\nu}$ is inner in $\C_+$ by \cite[Lemma 2.1]{Su19_1}.  
This assumption is satisfied unconditionally for $\omega>1/2$,   
and also for all $\omega>0$ under ${\rm GRH}(L)$ by Proposition \ref{prop_2_2} 
(see also Proposition \ref{prop_4_3}). 
We denote by ${\mathsf K}_L^{\omega,\nu}$ 
the isometry on $L^2(\R)$ defined by \eqref{s302} 
(cf.  \cite[Lemma 2.2]{Su19_1}. 
Note that it is not obvious whether the integral defines an operator on $L^2(\R)$ 
if we do not assume that $E_L^{\omega,\nu} \in \overline{\mathbb{HB}}$.)
%
If $(\omega,\nu)$ satisfies \eqref{s208}, 
we have 
\begin{equation*}
\mathsf{K}_L^{\omega,\nu}[t]  = {\mathsf P}_t{\mathsf K}_L^{\omega,\nu}{\mathsf P}_t
\end{equation*}
for $\mathsf{K}_L^{\omega,\nu}[t]$ in \eqref{s209} 
and the orthogonal projection ${\mathsf P}_t$ from $L^2(\R)$ to $L^2(-\infty,t)$. 

\begin{lemma} \label{lem_4_3} 
Let $L \in \mathcal{S}_\R$ and $\omega>0$. 
Then there exist entire functions $f_1^\omega(s)$ and $f_2^\omega(s)$ 
such that they have no common zeros, satisfy
\begin{equation*}
\frac{\xi_L(s-\omega)}{\xi_L(s+\omega)} = \frac{f_2^\omega(s)}{f_1^\omega(s)},
\end{equation*}
and the number of zeros of $f_2^\omega(s)$ in $|\Im(s)| \leq T$ 
is approximated by $c \, T \log T$ for large $T>0$, 
where $c>0$ is some constant. 
\end{lemma}
\begin{proof} 
We denote by $\mathcal{Z}_L$ the set of all zeros of $\xi_L(s)$ 
and by $m(\rho)$ the multiplicity of a zero $\rho \in \mathcal{Z}_L$. 
Then any zero of $\xi_L(s-\omega)$ has the form $s=\rho+\omega$ for some $\rho \in \mathcal{Z}_L$ 
and has the multiplicity $m(\rho)$. 
On the other hand, if $s=\rho+\omega$ for some $\rho \in \mathcal{Z}_L$ and $\xi_L(s+\omega)=0$, 
we have $\rho+2\omega \in \mathcal{Z}_L$. 
Considering this, we set
\begin{equation*}
\mathcal{Z}_L^\omega = \{ \rho \in \mathcal{Z}_L\,|\, \rho+2\omega \in \mathcal{Z}_L \},
\end{equation*}
and define an entire function by the Weierstrass product:
\begin{equation*}
\aligned
f_0^\omega(s) 
&= 
\prod_{\substack{\rho \in \mathcal{Z}_L^\omega \\ m(\rho) \geq m(\rho+2\omega)}}\left(1 - \frac{s}{\rho+\omega} \right)^{m(\rho+2\omega)} 
\exp\left( \frac{m(\rho+2\omega)\,s}{\rho-\omega} \right) \\
& \quad \times \prod_{\substack{\rho \in \mathcal{Z}_L^\omega \\ m(\rho) < m(\rho+2\omega)}}\left(1 - \frac{s}{\rho+\omega} \right)^{m(\rho)} 
\exp\left( \frac{m(\rho)\,s}{\rho-\omega} \right), 
\endaligned
\end{equation*}
where the right-hand side converges uniformly on compact subsets in $\C$, 
since $\xi_L(s)$ is an entire function of order one. 
In addition, we put
\begin{equation*}
f_1^\omega(s) = \frac{\xi_L(s+\omega)}{f_0^\omega(s)}, \quad 
f_2^\omega(s) = \frac{\xi_L(s-\omega)}{f_0^\omega(s)}.
\end{equation*}
Then, by definition, $f_1^\omega(s)$ and $f_2^\omega(s)$ are entire functions such that 
they have no common zeros and satisfy
\begin{equation*}
\frac{\xi_L(s-\omega)}{\xi_L(s+\omega)}=\frac{f_2^\omega(s)}{f_1^\omega(s)}.
\end{equation*}
Therefore, the remaining task is to show that $f_2^\omega(s)$ has approximately 
$c\,T \log T$ many zeros in $|\Im(s)| \leq T$ for some $c>0$. 

We denote by $N_L(T)$ (resp. $N_L^\omega(T)$) the number of zeros in   
$\mathcal{Z}_L$ (resp. $\mathcal{Z}_L^\omega$) with $|\Im(s)| \leq T$ counting with multiplicity: 
\begin{equation*}
N_L(T) = \sum_{\rho \in \mathcal{Z}_L,\, |\Im(s)| \leq T} m(\rho), \quad 
N_L^\omega(T) = \sum_{\rho \in \mathcal{Z}_L^\omega,\, |\Im(s)| \leq T} m(\rho)
\end{equation*}
and define
\begin{equation*}
n_L^\omega(T) = \sum_{\substack{\rho \in \mathcal{Z}_L,\, |\Im(s)| \leq T \\ m(\rho) \geq m(\rho+2\omega)}} m(\rho+2\omega) 
+ \sum_{\substack{\rho \in \mathcal{Z}_L, \, |\Im(s)| \leq T \\ m(\rho) < m(\rho+2\omega)}} m(\rho),
\end{equation*}
where the first sum is zero if it is an empty sum. 
Then, $n_L^\omega(T) \leq N_L^\omega(T) \leq N_L(T)$ 
and the number of zeros of $f_2^\omega(s)$ in $|\Im(s)| \leq T$ is $N_L(T) - n_L^\omega(T)$. 
We recall that $N_L(T) \sim (d_L/\pi)\,T \log T$, and it is so for the number of zeros of $\xi_L(s-\omega)$ with $|\Im(s)| \leq T$, 
where $f(T) \sim g(T)$ means that $f(T)/g(T) \to 1$ as $T \to \infty$. 
Therefore, $N_L(T) - n_L^\omega(T) \sim c \,T \log T$ for some $c>0$ unless $N_L(T) \sim n_L^\omega(T)$. 

Now we prove that $N_L(T) \not\sim n_L^\omega(T)$ by contradiction. 
Suppose that $N_L(T) \sim n_L^\omega(T)$. 
Then, $N_L(T) \sim N_L^\omega(T)$, since 
$n_L^\omega(T) \leq N_L^\omega(T) \leq N_L(T)$. 
We put 
$\Sigma_L^{\omega} = \{ \rho \in \mathcal{Z}_L^\omega\,|\, 1-\rho \in \mathcal{Z}_L^\omega\}$ 
and 
\begin{equation*}
M_L^{\omega}(T) = \sum_{\substack{\rho \in \Sigma_L^{\omega},\,|\Im(\rho)| \leq T}} m(\rho).
\end{equation*}
Then, $M_L^\omega(T) \sim N_L(T)$, 
because $N_L(T) \sim N_L^\omega(T)$ and functional equations \eqref{s203} imply that 
$\mathcal{Z}_L^\omega$ is closed under $\rho \mapsto 1-\rho$ and $\rho \mapsto \bar{\rho}$ 
except for a relatively small subset counting with multiplicity. 
If we take a zero $\rho \in \Sigma_L^{\omega}$, 
then $1-\rho \in \mathcal{Z}_L^\omega$ by the definition of $\Sigma_L^{\omega}$, 
and thus $1-\rho+2\omega \in \mathcal{Z}_L$ by the definition of $\mathcal{Z}_L^\omega$. 
Therefore, $\rho-2\omega=1-(1-\rho+2\omega) \in \mathcal{Z}_L$ by the first functional equation of \eqref{s203}. 
As a consequence, $\rho \in \Sigma_L^{\omega}$ implies $\rho-2\omega \in \mathcal{Z}_L^\omega$. 
On the other hand, $M_L^\omega(T) \sim N_L^\omega(T) \sim N_L(T)$ shows that 
$\rho-2\omega \in \mathcal{Z}_L^\omega$ implies $\rho-2\omega \in \Sigma_L^{\omega}$ almost surely. 
Taken together, $\rho \in \Sigma_L^{\omega}$ implies $\rho-2\omega \in \Sigma_L^{\omega}$ almost surely 
and this process is continued repeatedly. 
However, it is impossible, because all zeros of $\xi_L(s)$ must lie in the critical strip. 
Hence, $N_L(T) \not\sim n_L^\omega(T)$. 
\end{proof}

\begin{lemma} \label{lem_4_4} 
Let $t \geq 0$. Suppose that $(\omega, \nu)$ satisfies \eqref{s208}. 
Then the support of ${\mathsf K}_L^{\omega,\nu}{\mathsf P}_{t}f$ is not compact 
for every $f \in L^2(\R)$ unless ${\mathsf K}_L^{\omega,\nu}{\mathsf P}_{t}f =0$. 
\end{lemma}
\begin{proof} We prove this by contradiction. 
Suppose that ${\mathsf K}_L^{\omega,\nu}{\mathsf P}_{t} f \not=0$ and has a compact support. 
Then ${\mathsf F}{\mathsf K}_L^{\omega,\nu}{\mathsf P}_{t} f$ is an entire function of exponential type by the Paley-Wiener theorem.  
On the other hand, we have
\begin{equation*}
{\mathsf F}{\mathsf K}_L^{\omega,\nu}{\mathsf P}_{t} f(z) 
= \Theta_L^{\omega,\nu}(z) \cdot {\mathsf F}{\mathsf P}_{t} f(-z)
\end{equation*}
by  \cite[Lemma 2.2]{Su19_1}. 
If we put $G(z)={\mathsf F}{\mathsf P}_{t} f(-z)/f_1^\omega\left(\tfrac{1}{2}-iz\right)^\nu$, 
then we have  
\begin{equation*}
{\mathsf F}{\mathsf K}_L^{\omega,\nu}{\mathsf P}_{t} f(z)  = f_2^\omega\left(\tfrac{1}{2}-iz\right)^\nu \cdot G(z), 
\end{equation*}
where $f_1^\omega$ and $f_2^\omega$ are functions in Lemma \ref{lem_4_3}.  
Here $G(z)$ is entire, because, by Lemma \ref{lem_4_3}, the zeros of the numerator $f_2^\omega\left(\tfrac{1}{2}-iz\right)^\nu$ 
of $\Theta_L^{\omega,\nu}(z)$
can not kill the zeros of the denominator $f_1^\omega\left(\tfrac{1}{2}-iz\right)^\nu$, 
which therefore must be killed by zeros of ${\mathsf F}{\mathsf P}_{a} f(-z)$. 
This allows $f_2^\omega\left(\tfrac{1}{2}-iz\right)^\nu$ to be factored out.

The entire function on the right-hand side has at least $c \,T \log T$ many zeros 
in the disk of radius $T$ around the origin as $T \to \infty$ for some $c>0$ by Lemma \ref{lem_4_3}. 
However all entire functions of exponential type have at most $O(T)$ zeros in the disk of radius $T$ around the origin, as $T \to \infty$, 
because of the Jensen formula (\cite[\S2.5 (15)]{Le96}). This is a contradiction. 

As the above, it is not necessary to assume that $\Theta$ is inner in $\C_+$ 
for Lemma \ref{lem_4_4}. 
\end{proof}

\begin{proposition} \label{prop_4_4}
Let $t \geq 0$. We have 
${\rm i)}$ ${\mathsf K}_L^{\omega,\nu}[t] f =0$ for every $f \in L^2(-\infty,-t)$, 
${\rm ii)}$ $\Vert {\mathsf K}_L^{\omega,\nu}[t] f \Vert \not= \Vert f \Vert$ 
for every $0 \not=f \in L^2(-\infty,t)$, 
and ${\rm iii)}$ $\Vert {\mathsf K}_L^{\omega, \nu}[t] \Vert < 1$. 
In particular, $1 \pm {\mathsf K}_L^{\omega,\nu}[t]$ are invertible operators on $L^2(-\infty,t)$ 
for every $t \geq 0$. 
\end{proposition}
\begin{proof} 
First, we note that ${\mathsf K}_L^{\omega,\nu} f$ is defined for every $f \in L^2(-\infty,t)$ 
by \eqref{s302} and \cite[Lemma 2.2]{Su19_1}). 
Because $\int_{-\infty}^{t} K_L^{\omega,\nu}(x+y)f(y)\,dy =0$ for $x<-t$ by (K3), 
we obtain i).   

To prove ii), 
it is sufficient to show $\Vert {\mathsf K}_L^{\omega,\nu}[t] f \Vert < \Vert f \Vert$ unless $f=0$, 
because  
$\Vert {\mathsf K}_L^{\omega,\nu}[t] \Vert \leq \Vert {\mathsf P}_t \Vert \cdot \Vert {\mathsf K}_L^{\omega,\nu} \Vert \cdot \Vert {\mathsf P}_t \Vert = 1$ 
by  \cite[Lemma 2.2]{Su19_1}, and 
$\Vert {\mathsf K}_L^{\omega,\nu}[t] f \Vert \leq \Vert {\mathsf K}_L^{\omega,\nu}[t] \Vert \cdot \Vert f \Vert \leq \Vert f \Vert$. 
Here $\Vert {\mathsf K}_L^{\omega,\nu}[t] f \Vert \not= \Vert f \Vert$ 
is equivalent to 
$\Vert {\mathsf P}_t {\mathsf K}_L^{\omega,\nu} f \Vert \not= \Vert f \Vert$, 
since ${\mathsf P}_t f = f$ for $f \in L^2(-\infty,t)$. 
Suppose that $\Vert {\mathsf P}_t {\mathsf K}_L^{\omega,\nu} f \Vert = \Vert f \Vert$ for some $0 \not= f \in L^2(-\infty,t)$.  
Then it implies $\Vert {\mathsf P}_t {\mathsf K}_L^{\omega,\nu} f \Vert= \Vert {\mathsf K}_L^{\omega,\nu} f\Vert$ 
by $\Vert {\mathsf K}_L^{\omega,\nu} f \Vert = \Vert f \Vert$. 
Therefore
\begin{equation*}
\int_{-\infty}^{t} |{\mathsf K}_L^{\omega,\nu} f(x)|^2 \, dx 
= \int_{-\infty}^{\infty} |{\mathsf K}_L^{\omega,\nu} f(x)|^2 \, dx.
\end{equation*}
Thus ${\mathsf K}_L^{\omega,\nu} f(x) = 0$ for almost every $x>t$. 
On the other hand, we have 
\begin{equation*}
{\mathsf K}_L^{\omega,\nu} f(x) = \int_{-\infty}^{t} K_L^{\omega,\nu}(x+y)f(y)\,dy 
= \int_{-x}^{t} K_L^{\omega,\nu}(x+y)f(y)\,dy =0 
\end{equation*}
for $x<-t$ by $f\in L^2(-\infty,t)$. 
Hence ${\mathsf K}_L^{\omega,\nu} f$ has compact support contained in $[-t,t]$. 
However, it is impossible for any $f\not=0$ by Lemma \ref{lem_4_4}. 
As the consequence $\Vert {\mathsf K}_L^{\omega,\nu}[t] f \Vert < \Vert f \Vert$ for every $0\not=f \in L^2(-\infty,t)$. 

Finally, we prove iii). As stated in Theorem \ref{thm_1}, 
${\mathsf K}_L^{\omega,\nu}[t]$ is a self-adjoint compact operator 
(because the Hilbert-Schmidt operator is compact). 
Therefore, ${\mathsf K}_L^{\omega,\nu}[t]$ has purely discrete spectrum which has no accumulation points except for $0$, 
and one of $\pm \Vert {\mathsf K}_L^{\omega,\nu}[t] \Vert$ is an eigenvalue of ${\mathsf K}_L^{\omega,\nu}[t]$. 
However, by ii), every eigenvalue of ${\mathsf K}_L^{\omega,\nu}[t]$ has an absolute value less than $1$. 
Hence $\Vert {\mathsf K}_L^{\omega,\nu}[t] \Vert <1$. 
\end{proof}

\section{Proof of Theorem \ref{thm_4}} \label{section_7} 
%
%
\subsection{Necessity} 
%

If we take $\nu>1/(\omega d_L)$ for each $\omega>0$, 
$E_L^{\omega,\nu}$ satisfies (K1)$\sim$(K4) by Proposition \ref{prop_4_1}. 
In addition, $\Theta_L^{\omega,\nu}$ is inner in $\C_+$ for every $(\omega,\nu) \in \R_{>0} \times \Z_{>0}$ 
under ${\rm GRH}(L)$ 
by Proposition \ref{prop_4_3}. Therefore, we obtain (K5) with $\tau=\infty$ by Proposition \ref{prop_4_4}. 
Thus (1), (2) and (3) of Theorem \ref{thm_4} hold. Finally, (4) follows from Theorem \ref{thm_3}.
\hfill $\Box$

%
\subsection{Sufficiency}
%

By (1), (2) and (3), we obtain Hamiltonians $H_L^{\omega_n,\nu_n}$ on $[0,\infty)$ 
and the family of solutions ${}^{\rm t}[A_L^{\omega_n,\nu_n}(t,z),B_L^{\omega_n,\nu_n}(t,z)]$, 
$t \geq 0$,  of the canonical system on $[0,\infty)$ 
associated with $H_L^{\omega_n,\nu_n}$ satisfying  
the initial condition 
\begin{equation*}
A_L^{\omega_n,\nu_n}(0,z)=A_L^{\omega_n,\nu_n}(z) 
\quad \text{and} \quad 
B_L^{\omega_n,\nu_n}(0,z)=B_L^{\omega_n,\nu_n}(z). 
\end{equation*}
By \cite[Proposition 2.4]{Su19_1}, (4) implies that 
$E_L^{\omega_n,\nu_n}(t,z)=A_L^{\omega_n,\nu_n}(t,z)-iB_L^{\omega_n,\nu_n}(t,z)$ 
belongs to $\overline{\mathbb{HB}}$ for every $t \geq 0$. 
In particular, 
\begin{equation*}
E_L^{\omega_n,\nu_n}(0,z) = A_L^{\omega_n,\nu_n}(z)-iB_L^{\omega_n,\nu_n}(z) = \xi_L(\tfrac{1}{2}+\omega_n-iz)^{\nu_n}
\end{equation*}
belongs to $\overline{\mathbb{HB}}$. 
That is, $|\xi_L(\tfrac{1}{2}-\omega_n-iz)^{\nu_n}/\xi_L(\tfrac{1}{2}+\omega_n-iz)^{\nu_n}|<1$ if $\Im(z)>0$. 
It implies that $E_L^{\omega_n,1}(z)=\xi_L(\tfrac{1}{2}+\omega_n-iz)$ belongs to $\overline{\mathbb{HB}}$. 
In particular, $\xi_L(\tfrac{1}{2}+\omega_n-iz)$ has no zeros in $\C_+$ for every $n$. 
This implies that $\xi_L(\tfrac{1}{2}-iz)$ has no zeros in $\C_+$. 
In fact, if $\xi_L(\tfrac{1}{2}-iz)$ has a zero $\gamma=u+iv \in \C_+$,  
there exists $n$ such that $\omega_n <v$ and $E_L^{\omega_n,1}$ has a zero in $\C_+$, 
since $\omega_n \to 0$ and 
\begin{equation*}
\xi_L(\tfrac{1}{2}-i(u+iv))=\xi_L(\tfrac{1}{2}+\omega_n-i(u+i(v-\omega_n))). 
\end{equation*}
This contradicts $E_L^{\omega_n,1} \in \overline{\mathbb{HB}}$ for every $n$. 
The functional equation implies $\xi_L(\tfrac{1}{2}-iz)$ has no zeros in $\C_-$. 
Hence, all zeros of $\xi_L(\tfrac{1}{2}-iz)$ are real. \hfill $\Box$ 

%
%
\subsection{A variant of Theorem \ref{thm_4}} \label{section_9_1}
%
%

By Proposition \ref{prop_4_3} and the above argument, 
we obtain the following variant of Theorem \ref{thm_4}. 

\begin{theorem} \label{thm_4v} 
Let $L \in \mathcal{S}_\R$ and $0<\omega_0<1/2$. 
Then $L(s)\not=0$ for $\Re(s)>\tfrac{1}{2}+\omega_0$ if and only if 
there exists a sequence $(\omega_n,\nu_n) \in \R_{>0} \times \Z_{>0}$, $n \geq 1$, such that  
\begin{enumerate}
\item $\omega_m < \omega_n$ if $m>n$ and $\omega_n \to \omega_0$ as $n \to \infty$, 
\item $\nu_n \omega_n d_L>1$, 
\item $H_{L}^{\omega_n,\nu_n}(t)$ extends to a Hamiltonian on $[0,\infty)$, that is, 
$\det(1 \pm \mathsf{K}_{L}^{\omega_n,\nu_n}[t])\not=0$ for every $t \geq 0$, and 
\item $\displaystyle{\lim_{t \to \infty}J_{L}^{\omega_n,\nu_n}(t;z,z)=0}$ for every fixed $z \in \C_+$.  
\end{enumerate}
\end{theorem}

%
\section{Spectral realization of zeros of $A_L^{\omega,\nu}$ and $B_L^{\omega,\nu}$} \label{section_8} 
%

In this part, we mention that 
the zeros of $A_L^{\omega,\nu}$ and $B_L^{\omega,\nu}$ can be regarded as eigenvalues of 
self-adjoint extensions of a differential operator 
for $\omega > 1/2$ unconditionally and for $0<\omega\leq 1/2$ under ${\rm GRH}(L)$. 

%
\subsection{Multiplication by the independent variable} 
%

For $E \in \overline{\mathbb{HB}}$, 
the de Branges space $\mathcal{B}(E)$ has the unbounded operator $({\mathsf M},{\rm dom}({\mathsf M}))$  
consisting of multiplication by the independent variable $({\mathsf M}f)(z)=zf(z)$ endowed with the natural domain 
${\rm dom}({\mathsf M})=\{f\in \mathcal{B}(E) \,|\, zf(z) \in \mathcal{B}(E)\}$. 
The multiplication operator ${\mathsf M}$ is symmetric and closed, 
satisfies $\mathsf{M}(F^\sharp)=({\mathsf M}F)^\sharp$ for $F \in {\rm dom}(\mathsf{M})$ 
and has deficiency indices $(1,1)$ (\cite[Proposition 4.2]{KW99}). 
The operator $\mathsf{M}$ has no eigenvalues (\cite[Corollary 4.3]{KW99}). 

In general, ${\rm dom}({\mathsf M})$ has codimension at most one. 
Hereafter, we suppose that ${\rm dom}({\mathsf M})$ has codimension zero, that is, ${\rm dom}({\mathsf M})$ is dense in $\mathcal{B}(E)$. 
Then, all self-adjoint extensions ${\mathsf M}_\theta$ of $\mathsf{M}$ are parametrized by $\theta \in [0,\pi)$ 
and their spectrum consists of eigenvalues only. 
The self-adjoint extension $\mathsf{M}_\theta$ 
is described as follows.  
We introduce 
\begin{equation*}
S_\theta(z) = e^{i \theta}E(z) - e^{-i\theta}E^\sharp(z) 
\end{equation*}
for $\theta \in [0,\pi)$. 
Then, the domain of $\mathsf{M}_\theta$ is defined by 
\begin{equation*} 
{\rm dom}(\mathsf{M}_\theta) 
= \left\{\left.
G_F(z) = \frac{S_\theta(w_0)F(z)-S_\theta(z)F(w_0)}{z-w_0} ~\right|~
F(z) \in \mathcal{B}(E)
\right\}
\end{equation*}
and the operation is defined by 
\begin{equation*}
{\mathsf M}_\theta G_F(z) 
= z \, G_F(z) + F(w_0)S_\theta(z),
\end{equation*}
where $w_0$ is a fixed complex number with $S_\theta(w_0)\not=0$ 
and ${\rm dom}(\mathsf{M}_\theta) $ does not depend on the choice of $w_0$. 
The set 
\begin{equation*}
\left\{\left.
F_{\theta,\gamma}(z)=\frac{S_\theta(z)}{z-\gamma} 
~\right|~ S_\theta(\gamma)=0
\right\} 
\end{equation*}
forms an orthogonal basis of $\mathcal{B}(E)$, 
and each $F_{\theta,\gamma}$ is an eigenfunction of ${\mathsf M}_\theta$ 
with the eigenvalue $\gamma$: 
\begin{equation*}
{\mathsf M}_\theta F_{\theta,\gamma} = \gamma F_{\theta,\gamma}
\end{equation*}
(\cite[Proposition 6.1, Theorem 7.3]{KW99}). 
We have $S_{\pi/2}(z) = 2i \, A(z)$ and $S_0(z) = -2i \, B(z)$ by definition.  
Therefore, 
$\{ A(z)/(z-\gamma)\,|\,A(\gamma)=0\}$ and $\{ B(z)/(z-\gamma)\,|\,B(\gamma)=0\}$ 
are orthogonal basis of $\mathcal{B}(E)$ 
consisting of eigenfunctions of ${\mathsf M}_{\pi/2}$ and ${\mathsf M}_0$, respectively. 

%
\subsection{Transform to differential operators} 
%

Considering the isometric isomorphism of the Hilbert spaces 
\begin{equation*}
\mathcal{B}(E) \to \mathcal{K}(\Theta) \to \mathcal{V}_0 \subset L^2(0,\infty),
\end{equation*}
we define the differential operator $\mathsf{D}$ on $\mathcal{V}_0$ by 
\begin{equation*}
\mathsf{D} = \mathsf{F}^{-1}\mathsf{M}\mathsf{F}, \quad 
{\rm dom}(\mathsf{D}) 
= \mathsf{F}^{-1}\mathsf{M}_{\frac{1}{E}}({\rm dom}(\mathsf{M})) 
= \{ f \in \mathcal{V}_0 \,|\, z(\mathsf{F}f)(z) \in \mathcal{K}(\Theta) \},
\end{equation*}
where $\mathsf{M}_{\frac{1}{E}}$ is the operator of multiplication by $1/E(z)$.  
Then we have $(\mathsf{D} f)(x) = i \frac{d}{dx}f(x)$ for $f \in C^1(\R) \cap L^2(0,\infty)$.  
All self-adjoint extensions of $\mathsf{D}$ are given by
\begin{equation*}
\mathsf{D}_\theta = \mathsf{F}^{-1}\mathsf{M}_\theta\mathsf{F}, \quad 
{\rm dom}(\mathsf{D}_\theta) 
= \mathsf{F}^{-1}\mathsf{M}_{\frac{1}{E}}({\rm dom}(\mathsf{M}_\theta)), \quad \theta \in [0,\pi). 
\end{equation*}
The set 
\begin{equation*}
\{ f_{\theta,\gamma}=ie^{-i\theta}\mathsf{F}^{-1}\mathsf{M}_{\frac{1}{E}}F_{\theta,\gamma}\,|\, S_\theta(\gamma)=0\}
\end{equation*}
forms an orthogonal basis of $\mathcal{V}_0$ 
consisting of eigenfunctions $f_{\theta,\gamma}$ of $\mathsf{D}_\theta$ for eigenvalues $\gamma$, 
where $ie^{-i\theta}$ is the constant for the simplicity of $f_{\theta,\gamma}$. 
We have 
\begin{equation*}
f_{\theta,\gamma}(x) 
= e^{-i\gamma x} \left( \mathbf{1}_{[0,\infty)}(x) - e^{-2i\theta} \int_{0}^{x} K(y) e^{i\gamma y}\, dy \right) 
\end{equation*}
by the direct calculation of $\mathsf{F}^{-1}\mathsf{M}_{\frac{1}{E}}F_{\theta,\gamma}$. 
This formula suggests that the eigenfunction $f_{\theta,\gamma}$ 
is an adjustment of the ``eigenfunction'' $e^{-i\gamma x}$ of $i\frac{d}{dx}$ in $\mathcal{V}_0$. 

\begin{theorem} \label{thm_8}
Let $L \in \mathcal{S}_\R$, $(\omega,\nu) \in \R_{>0} \times \Z_{>0}$. 
Suppose that $E_L^{\omega,\nu} \in \overline{\mathbb{HB}}$ so that the de Branges space $\mathcal{B}(E_L^{\omega,\nu})$ is defined.  
Then ${\rm dom}(\mathsf{M})$ is dense in $\mathcal{B}(E_L^{\omega,\nu})$.  
\end{theorem}
\begin{proof} Let $E=E_L^{\omega,\nu}$, $\Theta=E^\sharp/E$. 
The domain of $\mathsf{M}$ is not dense in $\mathcal{B}(E)$ 
if and only if 
\begin{equation} \label{s801}
\sum_{\gamma \in \C,\,\Theta(\gamma)=0} \Im(\gamma) < \infty.
\end{equation} 
by \cite[Theorem 29]{MR0229011} and \cite[Theorem A and Corollary 2]{MR1855436}. 
The condition \eqref{s801} 
means that the zeros of $E^\sharp$ appearing in the zeros of $\Theta$ 
(all of them in $\C_+$) are finitely many or tend to the real line 
sufficiently quickly (from the above) if ${\rm dom}(\mathsf{M})$ is not dense in $\mathcal{B}(E)$. 
If the number of such zeros is finite, 
$\Theta$ is a rational function.  
But it contradicts \eqref{s402} and \eqref{s403}. 
If $E^\sharp$ has infinitely many zeros appearing in the zeros of $\Theta$ 
and tend to the real line, 
the functional equation of $\xi_L$ implies that there exists a zero of $E^\sharp$ 
above the horizontal line $\Im(z)=\omega$. 
Hence, \eqref{s801} is impossible. 
\end{proof}

Theorem \ref{thm_8} suggests that the pair $(\mathcal{V}_0, \mathsf{D}_{\theta})$ is a P{\'o}lya-Hilbert space  
for $A_L^{\omega,\nu}$ if $\theta=\pi/2$ and for $B_L^{\omega,\nu}$ if $\theta=0$. 
Noting that $A_L^{\omega,1} \to \xi_L$ as $\omega \to 0$ if $\epsilon_L=+1$ 
and $B_L^{\omega,1} \to i\xi_L$ as $\omega \to 0$ if $\epsilon_L=-1$, 
the family of pairs $\{ (\mathcal{V}_0^{\omega,\nu(\omega)}, \mathsf{D}_{\theta}^{\omega,\nu(\omega)}) \}_{\omega>0}$ 
may be considered as a perturbation of the ``genuine P{\'o}lya-Hilbert space'' 
associated with $E(z)=\xi_L(s)+\xi_L^\prime(s)$, 
where $\nu(\omega) \to \infty$ as $\omega \to 0$ under \eqref{s208}. 

On the other hand, $\mathcal{V}_0$ is isomorphic to the quotient space $L^2(0,\infty)/\mathsf{K}(L^2(-\infty,0))$ 
by \cite[Lemma 4.1]{Su19_1}. 
This structure of $\mathcal{V}_0$ is similar to Connes' suggestion for the P{\'o}lya-Hilbert space as explained below.  

%
\subsection{Comparison with Connes' P{\'o}lya-Hilbert space} 
%

Connes~\cite{Connes99} suggests a candidate of the P{\'o}lya-Hilbert space  
by interpreting the critical zeros of the Riemann zeta function 
as an absorption spectrum as follows. 
Let $S({\Bbb R})_0$ be the subspace of the Schwartz space $S({\Bbb R})$ 
consisting of all even functions $\phi \in S({\Bbb R})$ 
satisfying $\phi(0)=(\mathsf{F}\phi)(0)=0$. 
For a function $\phi \in S({\Bbb R})_0$, 
the function $\mathsf{Z}\phi$ on $\R_+=(0,\infty)$ is defined by 
$(\mathsf{Z}\phi)(y) = y^{1/2}\sum_{n=1}^\infty \phi(ny)$. 
Then $\mathsf{Z}\phi$ is of rapid decay as $y \to +0$ and $y \to +\infty$. 
In particular, $\mathsf{Z}(S(\R)_0) \subset L^2(\R_+,dy/y)$. 
Then the ``orthogonal complement'' $L^2(\R_+,dy/y) \ominus \mathsf{Z}(S(\R)_0)$ 
is spanned by generalized eigenfunctions $y^{-i\gamma} (\log y)^k$, $0 \leq k < m_\gamma$, 
of the differential operator $iyd/dy$ 
attached to the critical zeros $1/2+i\gamma$ of the Riemann zeta function 
with multiplicity $m_\gamma$. That is,
$\displaystyle{ 
\left(iy d/dy, ~L^2(\R_+,dy/y) \ominus \mathsf{Z}(S(\R)_0)\right)
}$
forms a ``P{\'o}lya-Hilbert space''. 
The differential operator $iy d/dy$ may be regarded as 
the shift of the Hamiltonian 
$
(1/2)( y [-i\hbar\, d/dy] + [-i \hbar\, d/dy] y ) 
= - i \hbar (yd/dy + 1/2)
$
of the Berry--Keating model~\cite{BK99}. 

Rigorously, the above argument does not make sense, 
since $y^{-i\gamma} (\log y)^k$ are not members of $L^2(\R_+,dy/y)$ 
and $L^2(\R_+,dy/y)=\overline{\mathsf{Z}(S(\R)_0)}$. 
However, the above naive idea is justified by several manners (\cite{Connes99} and R. Meyer~\cite{MR2132868, MR2206883}), 
but some nice property such as the self-adjointness of the operator, the spectral realization of zeros, the Hilbert space structure 
is lost by known justification. 

Contrast with justifications so far, 
the family $\{ (\mathcal{V}_0^{\omega,\nu(\omega)}, \mathsf{D}_{\pi/2}^{\omega,\nu(\omega)}) \}_{\omega>0}$ 
justifies Connes' idea preserving the self-adjointness of the operator, the spectral realization of zeros 
and the Hilbert space structure by considering the perturbation family $A_\xi^{\omega,\nu}$ of $\xi$. 
\medskip

Major objects of the above naive model of Connes' idea  
correspond to objects attached to $(\mathcal{V}_0, \mathsf{D}_\theta)$ 
as follows under the changing of variables $y=e^x$: 
\begin{equation*} 
\aligned 
L^2(\R_+^\times,dy/y) ~ & \Leftrightarrow~L^2(0,\infty) \\
\mathsf{Z}(S(\R)_0) ~ & \Leftrightarrow~ \mathsf{K}(L^2(-\infty,0)) \\ 
{\rm Mellin}(\mathsf{Z}(S(\R)_0))=\zeta(\tfrac{1}{2}+s){\rm Mellin}(S(\R)_0) ~ & 
\Leftrightarrow~ \mathsf{F}(\mathsf{K}(L^2(-\infty,0))) = \Theta(z)\mathsf{F}(L^2(0,\infty))\\
iy\frac{d}{dy}~ & \Leftrightarrow~ \mathsf{D}_\theta \approx i\frac{d}{dx},
\endaligned 
\end{equation*} 
where ``${\rm Mellin}$'' means the usual Mellin transform 
and $\approx$ means ``is equal up to domain''. 

%
%
\section{Miscellaneous Remarks} \label{section_10} 
%
%

\noindent
(1) Concerning the size of Dirichlet coefficients $a_L(n)$, 
the polynomial bound $|a_L(n)| \ll n^A$ for some $A \geq 0$ is enough to prove Theorems \ref{thm_1}, \ref{thm_2} and \ref{thm_3}. 
In other words, the Ramanujan conjecture (S4) is not necessary to prove these theorems. 
Therefore, the method of Sections \ref{section_3} and \ref{section_4} 
about the construction of $H_L^{\omega,\nu}$ 
and ${}^{\rm t}[A_L^{\omega,\nu}(t,z),B_L^{\omega,\nu}(t,z)]$ 
can apply to more general $L$-functions, 
in particular, to $L$-functions associated 
with self-dual irreducible cuspidal automorphic representations of $GL_n(\mathbb{A}_\Q)$ 
with unitary central characters.
\medskip 

\noindent
(2) In contrast with the Ramanujan conjecture (S4), 
the Euler product (S5) is essential to the construction of $H_L^{\omega,\nu}$ 
and ${}^{\rm t}[A_L^{\omega,\nu}(t,z),B_L^{\omega,\nu}(t,z)]$ . 
In fact, the explicit formula of the kernel $K_L^{\omega,\nu}$ coming from (S5) 
was critical to proved that $E_L^{\omega,\nu}$ satisfies (K2) and (K3). 
It seems that it is not easy even to prove that $K_L^{\omega,\nu}$ is a function 
if we do not have (S5). 
It is an interesting problem to extend the construction of $H_L^{\omega,\nu}$ 
and ${}^{\rm t}[A_L^{\omega,\nu}(t,z),B_L^{\omega,\nu}(t,z)]$ 
to the class of $L$-data which is an axiomatic framework for $L$-functions introduced by A. Booker~\cite{Booker15}. 
Superficially, Booker's $L$-datum does not require the Euler product, 
but it is based on the Weil explicit formula of $L$-functions in the Selberg class. 
Roughly, the Weil explicit formula is a result of (S3) and (S5), 
but the theory of $L$-data suggests that 
the Weil explicit formula is more essential than (S3) and (S5). 
\medskip 

\noindent
(3)  By the Euler product (S5), $L \in \mathcal{S}$ is expressed as a product of local $p$-factors $L_p$, 
and often, there exists polynomial $F_p$ of degree at most $d_L$ for each prime $p$ 
such that $L_p(s)=1/F_p(p^{-s})$. 
The Ramanujan conjecture (S4) is understood 
as an analogue of the Riemann hypothesis for $F_p(p^{-s})$. 
The Hamiltonian $H_{L,p}$ attached to $F_p(p^{-s})$ 
was constructed in \cite{Su16} 
by using a way analogous to the method in Section \ref{section_3} 
if $F_p$ is a real self-reciprocal polynomial 
(for details, see \cite[Section 1, Section 7.6]{Su16}). 
It is an interesting problem to find a relation among 
the perturbation family of global Hamiltonians $H_L^{\omega,\nu}$, 
the family of local Hamiltonians $H_{L,p}$ 
and 
the conjectural Hamiltonian $H_L$ corresponding to $E(z)=\xi_L(s) + \xi_L^\prime(s)$. 
\medskip 

\noindent
(4) The method of \cite{Su16} for the construction of $H$ and ${}^{\rm t}[A(t,z), B(t,z)]$ for exponential polynomials $E$ 
is useful to observe $H_L^{\omega,\nu}$ by numerical calculation of computer for concrete given $L \in \mathcal{S}_\R$, 
because an entire function satisfies (K1) is approximated by exponential polynomials 
by approximating the Fourier integral by Riemann sums 
(cf. the final part of the introduction of \cite{Su16}). 
\medskip 

\noindent
(5) A sharp estimate of $K_L^{\omega,\nu}(x)$ for large $x>0$ is not necessary 
to prove the main results of this paper. 
In fact, we do not know 
the role of the behavior of $K_L^{\omega,\nu}(x)$ at $x=+\infty$ 
in the equivalent condition of Theorem \ref{thm_4}.  
\medskip

\noindent
(6) If we replace the condition (4) of Theorem \ref{thm_4} by
\begin{enumerate}
\item[(4')] $E_L^{\omega_n,\nu_n}(0)\not=0$ and 
$\displaystyle{\lim_{t \to \infty}(A_L^{\omega_n,\nu_n}(t,z), B_L^{\omega_n,\nu_n}(t,z))=(E_L^{\omega_n,\nu_n}(0),0)}$,
\end{enumerate}
we obtain a sufficient condition for ${\rm GRH}(L)$, since (4)' implies (4). 
It is ideal if this is also a necessary condition, but we have no plausible evidence to support the necessity of (4)'. 
On the contrary, 
it is not clear whether $\lim_{t \to +\infty}A_L^{\omega,\nu}(t,z)$ 
defines a functions of $z$ contrast with the fact 
$\lim_{t \to +\infty}B_L^{\omega,\nu}(t,z)=0$ 
under $\omega>1/2$ or ${\rm GRH}(L)$ (see \cite[Section 5.4]{Su19_1}).  
The limit behavior may be related to the arithmetic properties of $L(s)$ in a deep level, 
because we need information of all $q_L^{\omega,\nu}(n)$'s to understand it 
differ from the situation that we need only finitely many $q_L^{\omega,\nu}(n)$'s 
to understand $\mathsf{K}_L^{\omega,\nu}[t]$ for a finite range of $t \in \R$. 
We do not touch this problem further in this paper. 

%

%

\bigskip \noindent
\\
Department of Mathematics, 
School of Science, \\
Tokyo Institute of Technology \\
2-12-1 Ookayama, Meguro-ku, 
Tokyo 152-8551, JAPAN  \\
Email: {\tt msuzuki@math.titech.ac.jp}

\end{document}